\documentclass[12pt]{article}

\usepackage{amsmath,amsthm,amsfonts,amssymb}
\usepackage{sidecap}
\usepackage{float}
\usepackage{extarrows}
\usepackage{booktabs}
\usepackage{verbatim}
\usepackage{hyperref}
\usepackage{mathtools}
\usepackage{esint}
\usepackage[normalem]{ulem}
\usepackage[usenames]{color}
\let\isout\sout
\renewcommand{\sout}[1]{\ifmmode\text{\isout{\ensuremath{#1}}}\else\isout{#1}\fi}

\newtheorem{theorem}{Theorem}[section]
\newtheorem{lemma}[theorem]{Lemma}
\newtheorem{corollary}[theorem]{Corollary}

\theoremstyle{definition}

\newtheorem*{definition}{Definition}

\def\R{\mathbb{R}}
\def\u{{\bf u}}
\def\w{{\bf w}}

\def\vv{{\bf v}}
\def\d{{\bf d}}
\def\P{\mathcal{P}}
\def\PP{\mathbb{P}}

\def\f{{\bf f}}
\newcommand{\pa}{\partial}
\def\p{\phi}
\def\hu{ \widehat{{\bf u}}}
\def\hd{ \widehat{{\bf d}}}
\def\hP{ \widehat{P}}
\def\od{ \overline{{\bf d}}}
\def\tu{ \widetilde{{\bf u}}}
\def\td{ \widetilde{{\bf d}}}
\def\tP{ \widetilde{P}}
\DeclareMathOperator{\dv}{div}
\DeclareMathOperator{\supp}{spt}
\DeclareMathOperator*{\loc}{loc}

\numberwithin{equation}{section}
\everymath{\displaystyle}
\allowdisplaybreaks
\begin{document}

\title{Partial regularity of a nematic liquid crystal model with kinematic transport effects}

\author{
	Hengrong Du\footnote{Department of Mathematics, Purdue University, West Lafayette, IN 47906, USA} 
	\quad Changyou Wang\footnote{Department of Mathematics, Purdue University, West Lafayette, IN 47906,
	USA. Both authors are partially supported by NSF grant 1764417.}\\
}
\date{\today}
\maketitle

\begin{abstract}

In this paper, we will establish the global existence of a suitable weak solution to the Erickson--Leslie system modeling hydrodynamics of nematic liquid crystal flows with kinematic transports for molecules of various shapes in ${\R^3}$, which is  smooth away from a closed set of (parabolic) Hausdorff dimension
at most $\frac{15}{7}$.

\medskip

\noindent{\it Keywords}: Ericksen--Leslie system, partial regularity, kinematic transports.

\medskip

\noindent {\it MR (2010) Subject Classification}: 35B65; 35Q35.

\end{abstract}

\section{\bf Introduction}\label{section-1}
In this paper, we will study the simplified Ericksen--Leslie system modeling the hydrodynamics of nematic liquid crystals with variable degrees of orientation and  kinematic transports for molecules of various shapes: 
 $(\u,\d,P):{\R^3}\times (0,\infty)\to{\R^3}\times{\R^3}\times\R$ solves
\begin{equation}
  \left\{
  \begin{array}{l}
    \pa_t\u+\u\cdot \nabla\u+\nabla P=\nu \Delta \u-\lambda \nabla\cdot (\nabla\d\odot\nabla \d+ S_\alpha[\Delta\d-\f(\d), \d]), \\
    \nabla\cdot \u=0, \\
    \pa_t\d+\u\cdot \nabla\d-T_\alpha[\nabla\u, \d]=\gamma(\Delta\d -\f(\d)),
  \end{array}
  \right.
  \label{eqn:ELmodel}
\end{equation}
where $\u(x,t)$ represents the velocity field of the flow, $\d(x,t)$ is the macroscopic averaged orientation
field of the nematic liquid crystal modules, and $P$ stands for the pressure function. 
Here $\f(\d)=D_{\d}F(\d)=(|\d|^2-1)\d$ is the gradient of Ginzburg--Landau potential 
function $F(\d)=\frac14(1-|\d|^2)^2$.
Furthermore, 
\begin{align*}
	S_\alpha[\Delta \d-\f(\d), \d]&:=\alpha(\Delta\d-\f(\d))\otimes\d-(1-\alpha)\d\otimes(\Delta \d-\f(\d)), \\
	T_\alpha[\nabla\u, \d]&:=\alpha(\nabla\u)\d-(1-\alpha)(\nabla \u)^T\d,
\end{align*}
represents the Leslie stress tensor and the kinematic transport term respectively.
The parameter $\alpha\in[0,1]$ is the shape parameter of the liquid crystal molecule.
In particular,  $\alpha=0, \frac12,$ and $1$ corresponds to disc-like, spherical and rod-like molecule shape respectively (cf. \cite{de1993physics,Ericksen1990Variabledegree,ericksen1961conservation,jeffery1922motion}). The coefficient $\nu$ represents the fluid viscosity, $\lambda$ stands for the competition between kinetic energy and potential energy, and $\gamma$ reflects the molecular relaxation time. 

In the 1960's, Ericksen and Leslie proposed a comprehensive hydrodynamic theory of nematic liquid crystals (cf. \cite{Ericksen1962, Leslie1968}). Since then there has been a great deal of theoretical and experimental work devoted to the study of nematic liquid crystal flows. The first rigorous mathematical study for 
the simplified Ericksen--Leslie system, that is, \eqref{eqn:ELmodel} without $S_\alpha, T_\alpha$ terms, 
was made by Lin--Liu \cite{linliu1995nonparabolic}. Later in \cite{linliu1996partialregularity},
they established a partial regularity for the suitable weak solutions which satisfy the local energy inequality,
analogous to the Navier-Stokes equations by Caffarelli--Kohn--Nirenberg in \cite{CKN1982ns}. 
Very recently, the same type of regularity result was obtained for the co-rotational Beris--Edwards 
$Q$-tensor model by Du--Hu--Wang \cite{du2019suitable}. 

In this paper, we will construct a global-in-time suitable weak solution to \eqref{eqn:ELmodel}, which enjoys
a partial regularity that is slightly weaker than that of \cite{linliu1996partialregularity}. Besides its own interest,
we believe that this partial regularity may be helpful to investigate the un-corotational Beris--Edwards system
due to a similar structure of nonlinearities. There are two major difficulties in the analysis of \eqref{eqn:ELmodel}:
\begin{itemize}
\item First, as pointed out by \cite{WuXuLiu2012kinematic}, when $\alpha \neq \frac{1}{2}$  the stretching effect induced by $T_\alpha[\nabla\u, \d]$  leads to the loss of maximum principle for the director field $\d$, which plays an essential role in \cite{linliu1996partialregularity, du2019suitable}.  Here, inspired by \cite{GiaquintaGiusti1973, lin1998newproof}, we will prove an $\varepsilon_0$-regularity result by a blowing-up argument that  involves a decay estimate of renormalized $L^{3}$-norm of
both $|\nabla \d|$ and $|\u|$ and the mean oscillation of $\d$ in $L^{6}$ as well.  
\item Second, the presence of stress tensor $S_\alpha[\Delta \d-\f(\d), \d]$ brings an extra difficulty 
on the decay estimate of renormalized $L^{\frac{3}{2}}$-norm of the pressure function $P$.
While in the co-rotational regime, i.e., $\alpha=\frac{1}{2}$, we know that $S_{\frac{1}{2}}$ is anti-symmetric,
which significantly simplifies the analysis on pressure function (see \cite{du2019suitable}). 
\end{itemize}
We would like to mention that in a recent preprint \cite{koch2020partial}, G. Koch obtained a partial regularity 
theorem for certain weak solutions to the Lin--Liu model that may be weaker than suitable weak solutions
and may not obey the maximum principle, in which a smallness condition is imposed on normalized $L^6$-norm 
of $|\d|$.

Before stating our main results, we need to introduce

\subsection*{Some notations.} For $\u, \w\in {\R^3}$, $A, B\in \R^{3\times 3}$, we denote
\begin{equation*}
\u\cdot \w:=\sum_{i=1}^{3}\u_i\w_i, \quad A:B=\sum_{i, j=1}^{3}A_{ij}B_{ij}, \quad
(A\cdot\w)_j:= \sum_{i=1}^{3} A_{ij} \w_i.
\end{equation*}
and 
\begin{equation*}
\begin{split}
  (\u\otimes \w)_{ij}=\u_i \w_j,\quad (\nabla\d\odot\nabla\d)_{ij}=\sum_{k=1}^{3}\pa_i\d_k\pa_j\d_k,\\
  [(\nabla\u)\d]_i=\sum_{j=1}^{3}\pa_j \u_i\d_j, \quad [(\nabla \u)^T\d]_i=\sum_{i=1}^{3}\pa_i\u_j \d_j.
  \end{split}
\end{equation*}
Define
\begin{equation*}
  {\bf H}=\text{Closure of }\left\{ \u\in C_0^\infty({\R^3},{\R^3}):\nabla\cdot \u=0 \right\} \text{ in }L^2({\R^3}), 
\end{equation*}
and
\begin{equation*}
  {\bf V}=\text{Closure of }\left\{ \u\in C_0^\infty({\R^3}, {\R^3}): \nabla\cdot \u=0 \right\} \text{ in }H^1({\R^3}).
\end{equation*}
For $0\le k\le 5$, $\P^k$ denotes the $k-$dimensional Hausdorff measure on ${\R^3}\times \R$ with respect to the parabolic distance:
\begin{equation*}
  \delta( (x, t), (y, s))=\max\left\{ |x-y|, \sqrt{|t-s|} \right\},\ \ \forall (x, t), (y,s)\in {\R^3}\times \R.
\end{equation*}
We let $B_r(x)$ denote the ball in ${\R^3}$ with center $x$ and radius $r$. 
For $z=(x, t)\in{\R^3}\times\R_+$, denote $\PP_r(z):=B_r(x)\times[t-r^2, t]$, and
$$f_{z, r}=\fint_{\PP_{r}(z)}f:=\frac{1}{|\PP_r(z)|}\int_{\PP_r(z)}f dxdt $$
for any function $f$ on $\PP_r(z)$. 

Since the exact values of $\nu, \lambda, \gamma$ don't play roles in our analysis, we will assume 
\begin{equation*}
  \nu=\lambda=\gamma=1.
\end{equation*}
With the following identity
\begin{equation*}
\nabla\cdot (\nabla\d\odot\nabla\d)=\nabla\d\cdot \Delta \d+\nabla \big( \frac{1}{2}|\nabla\d|^2 \big), 
\ \ \nabla F(\d)=\nabla\d\cdot \f(\d),
\end{equation*}
the system \eqref{eqn:ELmodel} can also be written as
\begin{equation}
  \left\{
  \begin{array}{l}
    \pa_t \u+\u\cdot \nabla\u+\nabla P=\Delta \u-\nabla\d\cdot (\Delta\d-\f(\d))-\nabla\cdot S_\alpha[\Delta \d-\f(\d), \d], \\
    \nabla\cdot \u=0, \\
    \pa_t \d+\u\cdot \nabla\d-T_\alpha[\nabla \u, \d]=\Delta\d-\f(\d). 
  \end{array}
  \right.
  \label{eqn:ELmodel1}
\end{equation}
subject to the initial  condition
\begin{equation}
(\u, \d)|_{t=0}=(\u_0, \d_0)\quad \text{in}\ \  {\R^3}.
\label{eqn:initial}
\end{equation}
\begin{definition}
   A pair of functions $(\u, \d):{\R^3}\times(0, \infty)\to {\R^3}\times {\R^3}$ is a weak solution of \eqref{eqn:ELmodel1} and \eqref{eqn:initial}, if $(\u, \d)\in (L_t^\infty L_x^2\cap L_t^2 H_x^1)({\R^3}\times(0, \infty), {\R^3})\times (L_t^\infty H_x^1\cap L_t^2 H_x^2)({\R^3}\times(0,\infty), {\R^3})$, and for any $\p\in
  C_0^\infty({\R^3}\times[0,\infty), {\R^3})$ and $\psi\in C_0^\infty({\R^3}\times[0,\infty), {\R^3})$, with $\dv \p=0$ in ${\R^3}\times[0,\infty)$, it holds that
  \begin{equation}
  \begin{split}
    &\int_{{\R^3}\times(0, \infty)}[-\u\cdot \pa_t \p+\nabla\u\cdot \nabla\p-\u\otimes\u:\nabla\p-(\psi\cdot\nabla\d)\cdot(\Delta\d-\f(\d))]dxdt\\
    &+\int_{{\R^3}\times(0,\infty)}S_\alpha[\Delta\d-\f(\d),\d]:\nabla\p dxdt=\int_{\R^3}\u_0\cdot \p(x, 0)dx,
    \end{split}
    \label{eqn:weaku}
  \end{equation}
  \begin{equation}
  \begin{split}
    &\int_{{\R^3}\times(0, \infty)}[-\d\cdot \pa_t \psi+\nabla\d: \nabla \psi-\u\otimes \d:\nabla \psi+\f(\d)\cdot \psi]dxdt\\
    &-\int_{{\R^3}\times(0,\infty)}T_\alpha[\nabla\u, \d]\cdot\psi dxdt =\int_{\R^3}\d_0\cdot\psi(x, 0)dx.
    \end{split}
    \label{eqn:weaksol}
  \end{equation}
\end{definition}

The global and local energy inequalities for \eqref{eqn:ELmodel1} play the basic roles: 
for $t>0$,
\begin{eqnarray}
 && \int_{\R^3}\big(\frac{1}{2}\big( |\u|^2+|\nabla\d|^2 \big)+F(\d)\big)(x,t) dx
  + \int_{0}^t\int_{\R^3}\big(|\nabla\u|^2+|\Delta\d-\f(\d)|^2 \big)(x,s)dxds\nonumber\\
 && \le\int_{\R^3}^{}\big(\frac{1}{2}(|\u_0|^2+|\nabla\d_0|^2)+F(\d_0)\big)(x)dx,
  \label{eqn:globaleng}
\end{eqnarray}
\begin{eqnarray}
  &&\int_{\R^3}\big[\frac{1}{2}\left( |\u|^2+|\nabla\d|^2 \right)+F(\d)\big]\p(x,t) dx
  +\int_{0}^{t}\int_{\R^3}\left( |\nabla\u|^2+|\Delta\d|^2 +|\f(\d)|^2\right)\p(x,s) dxds\nonumber\\
 && \le \int_{0}^{t}\int_{\R^3}\big[\frac{1}{2}(|\u|^2+|\nabla\d|^2)(\pa_t\p+\Delta\p)+F(\d)\pa_t \p\big](x,s) dxds\nonumber\\
&&+\int_{0}^{t}\int_{\R^3}\big[\frac{1}{2}\left( |\u|^2+2P \right)\u\cdot \nabla\p+\nabla\d\odot\nabla\d:\u\otimes\nabla\p\big](x,s)dxds\nonumber\\
&&  +\int_{0}^{t}\int_{\R^3}\left( \nabla\d\odot\nabla\d-|\nabla\d|^2 I_3 \right):\nabla^2\p(x,s)dxds\nonumber\\
&&  +\int_{0}^{t}\int_{\R^3}S_\alpha[\Delta\d-\f(\d), \d]:\u\otimes\nabla\p(x,s) dxds\nonumber\\
&&  +\int_{0}^{t}\int_{\R^3}T_\alpha[\nabla\u, \d]\cdot(\nabla\p\cdot \nabla\d)(x,s)dxds\nonumber\\
&&  -\int_{0}^{t}\int_{\R^3}(\nabla\p\cdot\nabla\d)\cdot\f(\d)dxds-2\int_{0}^{t}\int_{\R^3}\nabla\f(\d):\nabla\d \p(x,s) dxds.
  \label{eqn:localeng}
\end{eqnarray}
provided $0\le \p \in C_0^\infty({\R^3}\times (0,t]).$

It should be noted that the following cancellation
\begin{equation}
	\int_0^t\int_{\R^3} S_\alpha[\d_1, \d_2]:\nabla \u\p dxds=\int_0^t\int_{\R^3}T_\alpha[\nabla \u, \d_2]\cdot \d_1 \p dxds
	\label{eqn:keycancell}
\end{equation}
play a critical role in the later analysis.  
\begin{definition}
	A weak solution $(\u, \d, P)\in (L_t^\infty L_x^2 \cap L_t^2 H_x^1)({\R^3}\times (0,\infty), {\R^3})\times (L_t^\infty H_x^1\cap L_t^2 H_x^2)({\R^3}\times(0,\infty), {\R^3})\times L^{\frac{5}{3}}({\R^3}\times (0, \infty))$ of \eqref{eqn:ELmodel1} is a suitable weak solution of \eqref{eqn:ELmodel1}, if in addition, $(\u, \d, P)$ satisfies the local energy inequalities \eqref{eqn:localeng}. 
\end{definition}
The main theorem of this paper concerns both the existence and partial regularity of 
suitable weak solutions to the simplified Ericksen--Leslie model. 
\begin{theorem}
  For any $\u_0\in {\bf H}, \d_0\in H^1({\R^3},{\R^3})$, there exists a global suitable weak solution $(\u, \d, P):{\R^3}\times\R_+\to {\R^3}\times {\R^3}\times \R$ of the simplified Ericksen--Leslie system \eqref{eqn:ELmodel1}
  and  \eqref{eqn:initial} such that
   \begin{equation*}
    (\u, \d)\in C^\infty({\R^3}\times(0, \infty)\setminus \Sigma), 
  \end{equation*}
  where $\Sigma\subset {\R^3}\times \R_+$ is a closed subset with $\mathcal{P}^{\sigma}(\Sigma)=0, \forall \sigma>\frac{15}{7}$.
  \label{thm:main}
\end{theorem}

This paper is organized as follows. In section 2, we will derive both the global and local energy inequality for smooth solutions of \eqref{eqn:ELmodel1} and \eqref{eqn:initial}. In section 3, we will demonstrate the construction of suitable weak solution. In Section 4, we will prove the $\varepsilon_0$-regularity criteria for the suitable weak solutions. In section 5, we will finish the proof of the Theorem \ref{thm:main}. 
\section{Global and local energy inequalities}
In this section, we will derive both the global and local energy equalities for smooth solutions to \eqref{eqn:ELmodel1}.
\begin{lemma}
	Let $(\u, \d)\in C^\infty({\R^3}\times[0,\infty), {\R^3}\times{\R^3})$ be a solution to the simplified Ericksen-Leslie system \eqref{eqn:ELmodel1}. Then it holds that 
	\begin{equation}
	\frac{d}{dt}\int_{\R^3}\frac{1}{2}\left( |\u|^2+|\nabla\d|^2 \right)+F(\d) dx+\int_{\R^3}|\nabla\u|^2+|\Delta\d-\f(\d)|^2 dx=0.
	\label{eqn:globaleq}
	\end{equation}
	\label{lemma:globaleq}
\end{lemma}
\begin{proof}
	The proof is standard. See for instance \cite{WuXuLiu2012kinematic, WuXuLiu2013GeneralELsystem}.
\end{proof}
\begin{lemma}
	Let $(\u, \d, P)\in C^\infty({\R^3}\times (0, \infty), {\R^3}\times{\R^3}\times\R)$ be a solution to \eqref{eqn:ELmodel1}. Then for all $0\le \p\in C_0^\infty({\R^3}\times(0,\infty))$, it holds
	\begin{eqnarray}
 && \frac{d}{dt}\int_{\R^3}\big[\frac{1}{2}\left( |\u|^2+|\nabla\d|^2 \right)+F(\d)\big]\p dx
  +\int_{\R^3}\left( |\nabla\u|^2+|\Delta\d|^2+|\f(\d)|^2 \right)\p dx\nonumber\\
&&= \int_{\R^3}[\frac{1}{2}(|\u|^2+|\nabla\d|^2)(\pa_t\p+\Delta\p)+F(\d)\pa_t\p] dx\nonumber\\
&&+\int_{\R^3}\big[\frac{1}{2}\left( |\u|^2+2P \right)\u\cdot \nabla\p+\nabla\d\odot\nabla\d:\u\otimes\nabla\p\big]dx\nonumber\\
&&+\int_{\R^3}\left( \nabla\d\odot\nabla\d-|\nabla\d|^2 I_3 \right):\nabla^2\p(x,s)dx\nonumber\\
&&+\int_{\R^3}S_\alpha[\Delta\d-\f(\d), \d]:(\u\otimes\nabla\p)(x,s) dx
+\int_{\R^3}T_\alpha[\nabla\u, \d]\cdot(\nabla\p\cdot \nabla\d)(x,s)dx\nonumber\\
&&-\int_{\R^3}\f(\d)\cdot(\nabla\p\cdot\nabla\d)dx-2\int_{\R^3}\nabla\f(\d):\nabla\d \p(x,s) dx.
	\end{eqnarray}
	\label{lemma:localeg}
\end{lemma}
\begin{proof}
	Multiplying the $\u$ equation in \eqref{eqn:ELmodel1} by $\u\p$, integrating over ${\R^3}$, and by integration by parts we obtain
	\begin{eqnarray}
	&&\frac{d}{dt}\int_{\R^3}\frac{1}{2}|\u|^2\p dx+\int_{\R^3}|\nabla\u|^2\p dx\nonumber\\
	&&=\int_{\R^3}[\frac{1}{2}|\u|^2(\pa_t\p+\Delta\p) +\frac{1}{2}(|\u|^2+2P)\u\cdot \nabla\p] dx\nonumber\\
	&&-\int_{\R^3}(\u\cdot \nabla\d)\cdot \Delta\d \p dx+\int_{\R^3}(\u\cdot\nabla\d)\cdot\f(\d)\p dx\nonumber\\
   && +\int_{\R^3}S_\alpha[\Delta\d-\f(\d), \d]:(\u\otimes\nabla\p) dx                                          
    +\int_{\R^3}S_\alpha[\Delta\d-\f(\d), \d]:\nabla\u \p dx
	\label{eqn:localueq}
	\end{eqnarray}
	By taking derivatives of $\d$ equation in \eqref{eqn:ELmodel1}, we have
	\begin{equation*}
	\pa_t \nabla\d+\nabla(\u\cdot \nabla \d)=\nabla(\Delta\d-\f(\d)+T_\alpha[\nabla\u, \d]). 
	\end{equation*}
	Then multiplying this equation by $\nabla\d\p$, integrating over ${\R^3}$, we get
	\begin{eqnarray}
	&&\frac{d}{dt}\int_{\R^3}\frac{1}{2}|\nabla\d|^2 \p dx+\int_{\R^3}|\Delta\d|^2 \p dx\nonumber\\
	&&=\int_{\R^3}\frac{1}{2}|\nabla\d|^2 \pa_t \p +\int_{\R^3}(\u\cdot \nabla\d)\cdot (\Delta\d\p+\nabla\p\cdot \nabla\d) dx\nonumber\\
	&&-\int_{\R^3} \Delta\d\cdot (\nabla\p\cdot \nabla\d) dx-\int_{\R^3}\nabla(\f(\d)): \nabla\d \p dx\nonumber\\
   && -\int_{\R^3}T_\alpha[\nabla\u, \d]\cdot (\nabla\p \cdot \nabla\d)dx
    -\int_{\R^3}T_\alpha[\nabla\u, \d]\cdot \Delta\d\p dx.
	\label{eqn:localdeq}
	\end{eqnarray}
It follows from direct calculations that
	\begin{equation}
	\begin{split}
	-\int_{\R^3}\Delta\d \cdot (\nabla\p \cdot \nabla\d)dx
	&=\int_{\R^3}\frac{1}{2}|\nabla\d|^2 \Delta\p dx
	+\int_{\R^3}(\nabla\d\odot\nabla\d-|\nabla\d|^2I_3):\nabla^2\p dx.
	\end{split}
	\label{eqn:localeq2}
	\end{equation}
	 Moreover, multiplying the $\d$ equations by $\f(\d)\p$, integrating over ${\R^3}$, we get
	\begin{eqnarray}
	&&\frac{d}{dt}\int_{\R^3} F(\d)\p dx+\int|\f(\d)|^2 \p dx
	=\int_{\R^3} (F(\d)\pa_t \p -(\u\cdot\nabla\d)\cdot \f(\d)\p) dx\nonumber\\
	&&+\int_{\R^3} T_\alpha[\nabla \u, \d]\cdot \f(\d)\p dx
	-\int_{\R^3} (\nabla\f(\d):\nabla\d \p+ (\nabla\p\cdot \nabla\d)\cdot\f(\d)) dx.
	\label{eqn:Fdlocal}
	\end{eqnarray}
	Hence, by adding \eqref{eqn:localueq},  \eqref{eqn:localdeq},  \eqref{eqn:localeq2} together, and applying \eqref{eqn:keycancell}, we get \eqref{eqn:localeng}.
\end{proof}
\section{Existence of suitable weak solutions}

In this section, we will follow the same scheme in \cite{CKN1982ns,du2019suitable} to construct a suitable weak solution to \eqref{eqn:ELmodel1}.

We introduce the so-called retarded mollifier $\Psi_\theta$ for $f:{\R^3}\times \R_+\to \R$, with $0<\theta<1$, 
\begin{equation*}
\Psi_\theta[f](x,t)=\frac{1}{\theta^4}\int_{\R^4}\eta\left( \frac{y}{\theta}, \frac{\tau}{\theta} \right)\tilde{f}(x-y, t-\tau)dyd\tau, 
\end{equation*}
where 
\begin{equation*}
\tilde{f}(x, t)=
\left\{
\begin{array}{ll}
f(x, t) & t\ge 0, \\
0& t<0,
\end{array}
\right.
\end{equation*}
and the mollifying function $\eta\in C_0^\infty(\R^4)$ satisfies
\begin{equation*}
\left\{
\begin{array}{l}
\eta\ge0 \text{ and }\int_{\R^4}\eta dxdt=1,\\
\supp \eta\subset\left\{ (x, t):|x|^2 <t, 1<t<2 \right\}.
\end{array}
\right.
\end{equation*}

It is easy to verify that for $\theta\in(0,1]$ and $0<T\le \infty$ that 
\begin{equation*}
  \begin{split}
    \dv \Psi_\theta[\u]&=0 \text{ if }\dv \u=0, \\
    \sup_{0\le t\le T}\int_{\R^3}|\Psi_\theta[\w]|^2(x,t)dx&\le C\sup_{0\le t\le T}\int_{\R^3}|\w|^2 (x, t)dx, \\
    \int_{{\R^3}\times[0,T]}|\nabla\Psi_{\theta}[\w]|^2(x, t)dxdt&\le C\int_{{\R^3}\times [0,T]}|\nabla \w|^2(x, t)dxdt.
  \end{split}
\end{equation*}
Now with the mollifier $\Psi_\theta[\w]\in C^\infty(\R^4)$, we introduce the approximate system of  \eqref{eqn:ELmodel1}:
\begin{equation}
  \left\{
  \begin{array}{l}
    \pa_t \u^\theta+\Psi_\theta[\u^\theta]\cdot \nabla\u^\theta+\nabla P^\theta=\Delta \u^\theta-\nabla \Psi_\theta[\d^\theta]\cdot(\Delta \d^\theta-\f(\d^\theta)
    \\\qquad-\nabla\cdot S_\alpha[\Delta\d^\theta-\f(\d^\theta)),\Psi_\theta[\d^\theta]], \\
    \nabla\cdot \u^\theta=0, \\
    \pa_t \d^\theta+\u^\theta\cdot \nabla \Psi_\theta[\d^\theta]-T_\alpha[\nabla\u^\theta,\Psi_\theta[\d^\theta]]=\Delta \d^\theta-\f(\d^\theta).
  \end{array}
  \right. \text{ in }{\R^3}\times (0,T)
  \label{eqn:approximate}
\end{equation}
subject to the initial and boundary condition \eqref{eqn:initial}. 

For a fixed large integer $N\ge 1$, set $\theta=\frac{T}{N}\in(0, 1]$, we want to find $(\u^\theta, \d^\theta, P^\theta)$ solving \eqref{eqn:approximate}. This amounts to solving a coupling system of a Stokes-like system for $\u$ and a semi-linear parabolic-like equation for $\d$ with smooth coefficients. For $m=0$, we have $\Psi_\theta[\u^\theta]=\Psi_\theta[\d^\theta]=0$, and the system \eqref{eqn:approximate} reduces to a decoupled system
\begin{equation}
\left\{
\begin{array}{l}
\pa_t \u^\theta+\nabla P^\theta=\Delta \u^\theta, \\
\nabla\cdot \u^\theta=0, \\
\pa_t \d^\theta=\Delta \d^\theta-\f(\d^\theta), \\
(\u^\theta, \d^\theta)\Big|_{t=0}=(\u_0, \d_0)
\end{array} \text{ in }{\R^3}\times [0,\theta].
\right.
\label{eqn:decouple}
\end{equation}
which can be solved easily by the standard theory. Suppose now that the \eqref{eqn:approximate} has been solved for some $0\le k<N-1$. We are going to solve \eqref{eqn:approximate} in the time interval $[k\theta, (k+1)\theta]$ with an initial data 
\begin{equation}
(\u, \d)\Big|_{t=k\theta}=\lim_{t\uparrow k\theta}(\u^\theta, \d^\theta)(\cdot , t) \text{ in }{\R^3}.
\label{}
\end{equation}
Then one can solve the coupling system \eqref{eqn:approximate} using the Faedo-Galerkin method. In fact, for a pair of smooth test functions $(\p, \psi)\in {\bf V}\times H^2({\R^3}, {\R^3})$, the weak formulation for \eqref{eqn:approximate} reads
\begin{eqnarray}
&&\frac{d}{dt}\int_{\R^3}\u^\theta \cdot \p dx+\int_{\R^3}(\Psi_\theta[\u^\theta]\cdot \nabla\u^\theta) \cdot \p dx+\int_{\R^3}\nabla\u^\theta:\nabla\p dx \nonumber\\
&&=-\int_{\R^3}(\p\cdot\nabla\Psi_\theta[\d^\theta])\cdot(\Delta\d^\theta-\f(\d^\theta)) dx
+\int_{\R^3}S_\alpha[\Delta\d^\theta-\f(\d^\theta), \Psi_\theta[\d^\theta]]:\nabla\p dx, 
\label{weakfor1}
\end{eqnarray}
and 
\begin{eqnarray}
&&\frac{d}{dt}\int_{\R^3}\nabla\d^\theta:\nabla\psi dx-\int_{\R^3}(\u^\theta\cdot \nabla\Psi_\theta[\d^\theta])\cdot \Delta \psi dx\nonumber\\
&&=
-\int_{\R^3}(\Delta\d^\theta-\f(\d^\theta))\cdot \Delta\psi dx
-\int_{\R^3}T_\alpha [\nabla\u^\theta, \Psi_\theta[\d^\theta]]\cdot\Delta\psi dx.
\label{weakfor2}
\end{eqnarray}
We can solve the ODE system \eqref{weakfor1}-\eqref{weakfor2} with test function $(\psi, \p)$ chosen to be the basis of ${\bf V}\times H^2({\R^3}, {\R^3})$
up to a short time interval $[k\theta, k\theta+T_0]$.  Multiplying the $\u^\theta$ equation in \eqref{eqn:approximate} by $\u^\theta$, and the $\d^\theta$ equation by $-\Delta \d^\theta+\f(\d^\theta)$, integrating over ${\R^3}$ and adding two equations together we obtain
\begin{equation}
  \frac{d}{dt}\int_{\R^3}\frac{1}{2}\left( |\u^\theta|^2+|\nabla\d^\theta|^2 \right)+F(\d^\theta) dx+\int_{\R^3}\left( |\nabla\u^\theta|^2+|\Delta\d^\theta-\f(\d^\theta)|^2 \right)dx=  0.
  \label{eqn:approxGlobaleng}
\end{equation}
Next we need a uniform bound on $(\u^\theta, \d^\theta, P^\theta)$ to pass the limit $\theta\to 0$ to get 
a suitable weak solution. First by direct calculations we can show that 
\begin{equation}
\begin{split}
  &\int_{\R^3}|\Delta\d^\theta-\f(\d^\theta)|^2dx=\int_{\R^3}[|\Delta \d^\theta|^2+|\f(\d^\theta)|^2-2 \Delta\d^\theta \cdot \f(\d^\theta)]dx\\
  &=\int_{\R^3}(|\Delta\d^\theta|^2+|\f(\d^\theta)|^2 +2\nabla\d^\theta: \nabla \f(\d^\theta))dx\\
  &=\int_{\R^3}(|\Delta \d^\theta|^2+|\f(\d^\theta)|^2-2|\nabla\d^\theta|^2+2|\nabla\d^\theta|^2|\d^\theta|^2+4|(\nabla\d^\theta)^T\d^\theta|^2)dx.
  \end{split}
  \label{2.6}
\end{equation}
From \eqref{eqn:approxGlobaleng}, we can obtain that
\begin{equation}
  \sup_{0<\theta<1}\sup_{0<t<T}\int_{\R^3}\left( |\u^\theta|^2+|\nabla\d^\theta|^2 \right)dx+\int_{0}^{T}\int_{\R^3}|\nabla\u^\theta|^2+|\Delta\d^\theta-\f(\d^\theta)|^2dxdt\le C(\u_0, \d_0). 
  \label{2.7}
\end{equation}
Combining \eqref{2.6} and \eqref{2.7},  we get
\begin{equation}
  \begin{split}
   & \int_{0}^{T}\int_{\R^3}|\Delta \d^\theta|^2+|\f(\d^\theta)|^2 dxdt\le \int_{0}^{T}\int_{\R^3}|\Delta \d^\theta-\f(\d^\theta)|^2 dxdt+2\int_{0}^{T}\int_{\R^3}|\nabla\d^\theta|^2 dxdt\\
    &\le \int_{0}^{T}\int_{\R^3}|\Delta\d^\theta-\f(\d^\theta)|^2dxdt+2T\sup_{0<t<T}\int_{\R^3}|\nabla\d^\theta|^2dx\\
    &\le C(\u_0, \d_0, T),
  \end{split}
  \label{2.8}
\end{equation}
From \eqref{2.7} and \eqref{2.8}, we have that $\u^\theta$ is uniformly bounded in $L_t^2 H_x^1({\R^3}\times[0,T])$, $\d^\theta$ is uniformly bounded in $L_t^2 H_x^2({\R^3}\times[0,T])$ for any compact set $K\subset {\R^3}$, and $\nabla\d^\theta$ is uniformly bounded in $L_t^2 H_x^1({\R^3}\times [0,T])$. Therefore, after passing to a subsequence, there exist $\u\in L_t^\infty L_x^2\cap L_t^2 H_x^1({\R^3}\times[0,T])$ and
 $\d\in L_t^\infty H^1_x\cap L_t^2 H_x^2({\R^3}\times[0,T])$ such that 
\begin{equation}
  \left\{
  \begin{array}{ll}
    \u^\theta\rightharpoonup \u & \text{ in }L_t^\infty L_x^2 \cap L_t^2 H_x^1({\R^3}\times[0,T]), \\
    \d^\theta\rightharpoonup \d & \text{ in }L_t^\infty H_x^1 \cap L_t^2 H_x^2({\R^3}\times[0,T]), \\
    \f(\d^\theta)\rightharpoonup \f(\d) & \text{ in }L_t^2 L_x^2 ({\R^3}\times [0,T]).
  \end{array}
  \right.
  \label{}
\end{equation}
By the Sobolev-interpolation inequality, we have that $\nabla\d^\theta\in L_t^{10}L_x^{\frac{30}{13}}, \d^\theta\in L_t^{10}L_x^{10}$, and 
\begin{equation}
  \begin{split}
    &\int_{0}^{T}\left\|\nabla\d^\theta\right\|_{L_x^{\frac{30}{13}}}^{10} dt\le \int_{0}^{T}\left\|\nabla \d^\theta\right\|_{L_x^2}^{8}\left\|\nabla \d^\theta\right\|_{L_x^6}^2dt\le  \left\|\d^\theta\right\|_{L_t^\infty H_x^1}^8 \left\| \d^\theta\right\|_{L_t^2 H_x^2}^2<\infty, \\
    &\int_{0}^{T}\left\|\d^\theta\right\|_{L_x^{10}}^{10} dt\le \int_0^T \left\|\d^{\theta}\right\|_{W_x^{1, \frac{30}{13}}}^{10} dt<\infty.  \end{split}
  \label{eqn:L10interpolation}
\end{equation}
By the lower semicontinuity and \eqref{eqn:approxGlobaleng},  we have, 
for $E(\u, \d)=\int_{\R^3}\frac{1}{2}(|\u|^2+|\nabla\d|^2+F(\d))dx$, that 
\begin{equation}
    E(\u, \d)(t)+\int_{0}^{t}\int_{\R^3}(|\nabla\u|^2+|\Delta\d-\f(\d)|^2)dxdt\le E(\u_0, \d_0)
      \label{}
\end{equation}
holds for a.e.  $0\le t\le T$. 

Now we want to estimate the pressure function $P^\theta$. Taking the divergence of $\u^\theta$ equation in \eqref{eqn:approximate} gives
\begin{equation}
\begin{split}
  -\Delta P^\theta&=\dv^2(\Psi_\theta[\u^\theta]\otimes \u^\theta)+\dv\left( \nabla(\Psi_\theta[\d^\theta])\cdot (\Delta\d^\theta-\f(\d^\theta)) \right)\\
  &\ \ \ +\dv^2\left[ S_\alpha[\Delta\d^\theta-\f(\d^\theta), \Psi_\theta[\d^\theta]] \right], \quad 
  \text{ in }{\R^3}.
  \end{split}
  \label{eqn:approxPeq}
\end{equation}
For $P^\theta$, we claim that $P^\theta$ in $L^{\frac{5}{3}}({\R^3}\times [0,T])$ and 
\begin{equation*}
  \left\|P^\theta\right\|_{L^{\frac{5}{3}}({\R^3}\times[0,T])}\le C(\left\|u_0\right\|_{L^2({\R^3})}, \left\|\d_0\right\|_{H^1({\R^3})}, T), \ \forall \theta\in(0,1].
\end{equation*}
In fact, by Calderon-Zgymund's $L^p$-theory, we have
\begin{align*}
  \left\|P^\theta\right\|_{L^{\frac{5}{3}}({\R^3}\times[0,T])}&\le C
  \Big[ \left\|\Psi_{\theta}[\u^\theta]\otimes \u^\theta\right\|_{L_t^{\frac{5}{3}}L_x^{\frac{5}{3}}}+\left\|\nabla(\Psi_\theta[\d^\theta])\cdot (\Delta\d^\theta-\f(\d^\theta))\right\|_{L_t^{\frac{5}{3}}L_x^{\frac{15}{14}}}\\&\qquad\qquad+ \left\||\Psi_\theta[\d^\theta]||\Delta\d^\theta-\f(\d^\theta)|\right\|_{L_t^{\frac{5}{3}}L_x^{\frac{5}{3}}}\Big]\\
 &\le C \Big[ \left\|\u^\theta\right\|_{L_t^{\frac{10}{3}}L_x^{\frac{10}{3}}}^2 + \left\|\nabla\d^\theta\right\|_{ L_t^{10} L_x^{\frac{30}{13}} }  \left\|\Delta\d^\theta-\f(\d^\theta)\right\|_{L_t^2 L_x^2}\\
 &\qquad\qquad+
 \left\|\d^\theta\right\|_{L_t^{10} L_x^{10}}\left\|\Delta\d^\theta-\f(\d^\theta)\right\|_{L_t^2 L_x^2}\Big]
 \\
 &\le C(\left\|\u\right\|_{L_t^\infty L_x^2\cap L_t^2 H_x^1({\R^3}\times[0,T])}, \left\|\d\right\|_{L_t^\infty H_x^1 \cap L_t^2 H_x^2({\R^3}\times [0,T])})\\
 &\le C(\left\|\u_0\right\|_{L^2({\R^3})}, \left\|\d_0\right\|_{H^1({\R^3})}, T).
\end{align*}
This uniform estimate implies that there exists $P\in L^{\frac{5}{3}}({\R^3}\times [0,T])$ such that as $\theta\to 0$, 
\begin{equation}
  P^\theta\rightharpoonup P \text{ in }L^{\frac{5}{3}}({\R^3}\times[0,T]). 
  \label{eqn:weakconvP}
\end{equation}
Recalling the $\u^\theta$ equation, we get 
\begin{equation*}
  \begin{split}
    \pa_t \u^\theta&=-\Psi_\theta[\u^\theta]\cdot \nabla\u^\theta-\nabla P^\theta+\Delta\u^\theta-\nabla(\Psi_\theta[\d^\theta])\cdot (\Delta \d^\theta-\f(\d^\theta))\\
    &\qquad-\nabla\cdot S_\alpha[\Delta\d-\f(\d^\theta), \Psi_\theta[\d^\theta]]\\
    &\in L^{\frac{5}{4}}({\R^3}\times[0,T])+L^{\frac{5}{3}}([0,T], W^{-1, \frac{5}{3}}({\R^3}))+ \bigcap_{R>0}L^2([0,T], W^{-1, \frac{3}{2}}(B_R)), 
  \end{split}
\end{equation*}
and
\begin{equation*}
\begin{split}
  \sup_{0<\theta<1}\left\|\pa_t \u^\theta\right\|_{L^{\frac{5}{4}}({\R^3}\times[0,T])+L^{\frac{5}{3}}([0,T], W^{-1, \frac{5}{3}}({\R^3}))+L^2([0,T], W^{-1, \frac{3}{2}}(B_R))}\\
  \le C(R,T, \left\|\u_0\right\|_{L^2({\R^3})}, \left\|\d_0\right\|_{H^1({\R^3})}).
  \end{split}
\end{equation*}
Similarly, we can show
\begin{equation*}
  \pa_t \d^\theta\in L^{\frac{5}{3}}({\R^3}\times[0,T])+\bigcap_{R>0}L^2([0,T], L^{\frac{4}{3}}(B_R)),
\end{equation*}
and
\begin{equation*}
  \left\|\pa_t \d^\theta\right\|_{L^{\frac{5}{3}}({\R^3}\times[0,T])+\bigcap_{R>0}L^2([0,T], L^{\frac{3}{2}}(B_R))({\R^3}\times[0,T])}\le C(R, T, \left\|\u_0\right\|_{L^2({\R^3})}, \left\|\d_0\right\|_{H^1({\R^3})}).
\end{equation*}
Hence by the Sobolev embedding and Aubin--Lions' compactness Lemma, we can conclude that as $\theta\to 0$,
\begin{equation}
  \left\{
  \begin{array}{ll}
    \u^\theta\to \u & \text{ in }L^{p_1}({\R^3}\times[0,T]), 1<p_1<\frac{10}{3}, \\
    \nabla\u^\theta\rightharpoonup \nabla \u & \text{ in }L^2({\R^3}\times[0,T]), \\
    \d^\theta\to \d & \text{ in }L^{p_2}({\R^3}\times[0,T]), 1<p_2<10, \\
    \nabla\d^\theta\to \nabla \d & \text{ in }L^{p_1}({\R^3}\times[0,T]), 1<p_1<\frac{10}{3}, \\
    \nabla^2\d^\theta \rightharpoonup \nabla^2 \d & \text{ in }L^2({\R^3}\times[0,T]).
  \end{array}
  \right.
  \label{eqn:approximateconv}
\end{equation}
Furthermore, $(\u^\theta, \d^\theta, P^\theta)$ satisfies the local energy inequality. In fact, if we multiply the $\u^\theta$ equation in \eqref{eqn:approximate} by $\u^\theta\p$,  take derivative of the $\d^\theta$ equation
in \eqref{eqn:approximate} and multiply by $\nabla\d^\theta \p$,  multiply the $\d^\theta$ equation in \eqref{eqn:approximate} by $\f(\d^\theta)$, and perform calculations similar to the previous section,
we can get
	\begin{equation}
	\begin{split}
&\frac{d}{dt}\int_{\R^3}\big[\frac{1}{2}\left( |\u^\theta|^2+|\nabla\d^\theta|^2 \right)+F(\d^\theta)\big]\p dx
+\int_{\R^3}\left( |\nabla\u^\theta|^2+|\Delta\d^\theta|^2+|\f(\d^\theta)|^2 \right)\p dx\\
&= \int_{\R^3}[\frac{1}{2}(|\u^\theta|^2+|\nabla\d^\theta|^2)(\pa_t\p+\Delta\p)+F(\d^\theta)\pa_t\p] dx\\
&+\int_{\R^3}\big[\frac{1}{2} |\u^\theta|^2 \Psi_{\theta}[\u^\theta]\cdot \nabla\p+P^\theta \u^\theta\cdot\nabla\p+\nabla\Psi_{\theta}[\d^\theta]\odot\nabla\d^\theta:\u^\theta\otimes\nabla\p\big]dx\\
&+\int_{\R^3}\left( \nabla\d^\theta\odot\nabla\d^\theta-|\nabla\d^\theta|^2 I_3 \right):\nabla^2\p dx\\
&+\int_{\R^3}S_\alpha[\Delta\d^\theta-\f(\d^\theta),\Psi_\theta[\d^\theta]]:(\u^\theta\otimes\nabla\p) dx
+\int_{\R^3}T_\alpha[\nabla\u^\theta, \Psi_\theta[\d^\theta]]\cdot(\nabla\p\cdot \nabla\d^\theta)dx\\
&-\int_{\R^3}\f(\d^\theta)\cdot(\nabla\p\cdot\nabla\d^\theta)dx
-2\int_{\R^3}\nabla\f(\d^\theta):\nabla\d^\theta \p dx.
\end{split}
\label{eqn:approxlocalenergy}
\end{equation}
With the convergence  \eqref{eqn:weakconvP}, \eqref{eqn:approximateconv}, it is easy to check that the limit $(\u, \d)$ is a weak solution to \eqref{eqn:ELmodel1} and \eqref{eqn:initial}. Taking the limit in \eqref{eqn:approxlocalenergy} as $\theta\to 0$, by the lower semicontinuity we obtain
\begin{equation}
\begin{split}
 & \int_{\R^3}\left[ \frac{1}{2}(|\u|^2+|\nabla\d|^2)+F(\d) \right]\p (x, t)dx
  +\int_{0}^{t}\int_{\R^3}(|\nabla\u|^2+|\Delta\d|^2+|\f(\d)|^2)\p dxds\\
  &\le\liminf_{\theta\to0}\int_{\R^3}\left[ \frac{1}{2}(|\u^\theta|^2+|\nabla\d^\theta|^2)+F(\d^\theta) \right]\p (x, t)dx\\
 &\ \ \  +\int_{0}^{t}\int_{\R^3}(|\nabla\u^\theta|^2+|\Delta\d^\theta|^2+|\f(\d^\theta)|^2)\p dxds\Big].
  \end{split}
\label{}
\end{equation}
While 
\begin{equation}
  \begin{split}
    &\lim_{\theta\to 0} \text{R.H.S. of }\eqref{eqn:approxlocalenergy}\\
    &=\int_{\R^3}\frac{1}{2}(|\u|^2+|\nabla\d|^2)(\pa_t\p+\Delta\p)+F(\d)\pa_t\p dx\\
    &+\int_{\R^3}\left[\frac{1}{2}\left( |\u|^2+2P \right)\u\cdot \nabla\p+\nabla\d\odot\nabla\d:\u\otimes\nabla\p\right]dx\\
    &+\int_{\R^3}\left( \nabla\d\odot\nabla\d-|\nabla\d|^2 I_3 \right):\nabla^2\p dx\\
    &+\int_{\R^3}S_\alpha[\Delta\d-\f(\d),\d]:(\u\otimes\nabla\p) dx\\
    &+\int_{\R^3}T_\alpha[\nabla\u,\d]\cdot(\nabla\p\cdot \nabla\d) dx\\
    &-\int_{\R^3}\f(\d)\cdot(\nabla\p\cdot\nabla\d)dx
    -2\int_{\R^3}\nabla\f(\d):\nabla\d \p dx
  \end{split}
  \label{}
\end{equation}
Putting all those together we show that the local energy inequality \eqref{eqn:localeng} holds. Therefore 
$(\u, \d, P)$ is a suitable weak solution to \eqref{eqn:ELmodel1} and \eqref{eqn:initial}.

\section{$\varepsilon_0$-Regularity criteria}
In this section we will establish the partial regularity for suitable weak solutions $(\u, \d, P)$ of \eqref{eqn:ELmodel1} in ${\R^3}\times(0, \infty)$. The argument is based on a blowing up argument, motivated by that of Lin \cite{lin1998newproof} on the Navier--Stokes equation. Recently, this type of argument
has been employed by Du--Hu--Wang \cite{du2019suitable} for the partial regularity in the co-rotational Beris--Edwards system in dimension three. 
However, the kinematic transport effects in \eqref{eqn:ELmodel1} destroy the maximum principle  
for $\d$, which is necessary to apply the argument by \cite{lin1998newproof} and \cite{du2019suitable}. 
To overcome this new difficulty, we adapt some ideas from
Giaquinta--Giusti \cite{GiaquintaGiusti1973} to control the mean oscillation of $\d$
in $L^{6}$. More precisely, we have
\begin{lemma}
  For any $M>0$, there exist $\varepsilon_0=\varepsilon_0(M)>0$, $0<\tau_0(M)<\frac{1}{2}$, and $C_0=C_0(M)>0$,  such that if $(\u, \d, P)$ is a suitable weak solution  of \eqref{eqn:ELmodel1} in ${\R^3}\times(0,\infty)$, which satisfies, for $z_0=(x_0, t_0)\in {\R^3}\times(r^2, \infty)$ and $r>0$, 
  \begin{equation}
    |\d_{z_0, r}|:=\big|\fint_{\PP_r(z_0)}\d dxdt\big|\le M, 
    \label{}
  \end{equation}
  and
  \begin{equation}
  \begin{split}
    \Phi(z_0, r):&= r^{-2}\int_{\PP_{r}(z_0)}\left( |\u|^{3}+|\nabla\d|^{3} \right)dxdt
    +\Big( r^{-3}\int_{\PP_r(z_0)}|P|^{\frac{3}{2}}dxdt \Big)^{2}\\
    &\ \ \ \ +\Big(\fint_{\PP_{r}(z_0)}|\d-\d_{z_0, r}|^{6} dxdt\Big)^{\frac{1}{2}}\le \varepsilon_0^{3},
   \end{split}
    \label{}
  \end{equation}
  then 
  \begin{equation}
      \Phi(z_0, \tau_0 r)\le \frac{1}{2}\max\Big\{ \Phi(z_0, r), C_0r^{3} \Big\}.
    \label{}
  \end{equation}
  \label{lemma:smallregularity}
\end{lemma}
\begin{proof}
  We prove it by contradiction. Suppose that the conclusion were false. Then there exists $M_0>0$ such that for any $\tau\in(0, \frac{1}{2})$, there exists
  $\varepsilon_i\to 0, C_i\to \infty$, and $r_i>0$, and $z_i=(x_i, t_i)\in{\R^3}\times(r_i^2, \infty)$ such that 
  \begin{equation}
    |\d_{z_i, r_i}|\le M_0,
    \label{4.4}
  \end{equation}
  and 
  \begin{equation}
    \begin{split}
      \Phi(z_i, r_i)&=\varepsilon_i^{3}, 
    \end{split}
    \label{4.5}
  \end{equation}
  but
  \begin{equation}
    \begin{split}
      \Phi(z_i, \tau r_i)&\ge \frac{1}{2}\max\left\{ \varepsilon_i^{3}, C_i r_i^{3} \right\}, 
    \end{split}
    \label{3.7}
  \end{equation}
  Notice that 
  \begin{equation*}
  \begin{split}
   & (\tau r_i)^{-2}\int_{\PP_{\tau r_i}(z_i)}\left( |\u|^{3}+|\nabla\d|^{3} \right)dxdt
    +\Big( (\tau r_i)^{-2}\int_{\PP_{\tau r_i}(z_i)}|P|^{\frac{3}{2}}dxdt \Big)^{2}\\
   & \le \tau^{-4}\Big( r_i^{-2}\int_{\PP_{r_i}(z_i)} \left( |\u|^{3}+|\nabla\d|^{3} \right)dxdt+\big( r_i^{-2}\int_{\PP_{r_i}(z_i)} |P|^{\frac{3}{2}}dxdt\big)^{2}\Big),
   \end{split}
    \label{}
  \end{equation*}
  \begin{equation*}
  \begin{split}
   & \Big( \fint_{\PP_{\tau r_i}(z_i)}|\d-\d_{z_i, \tau r_i}|^{6} dxdt\Big)^{\frac{1}{2}}
    \le\Big( 2^{5} |\d_{z_i,\tau r_i}-\d_{z_i, r_i}|^{6}+ 2^{5}\fint_{\tau r_i (z_i)}|\d-\d_{z_i, r_i}|^{6}dxdt \Big)^{\frac{1}{2}}\\
    &=\Big( 2^{5} \Big|\fint_{\PP_{\tau r_i}(z_i)}(\d-\d_{z_i, r_i})dxdt\Big|^{6}+2^{5}\fint_{\tau r_i (z_i)}|\d-\d_{z_i, r_i}|^{6}dxdt \Big)^{\frac{1}{2}}\\
    &\le\Big( 2^{6}\fint_{\PP_{\tau r_i}(z_i)}|\d-\d_{z_i, r_i}|^{6} dxdt \Big)^{\frac{1}{2}}
    \le 2^{3}\tau^{-{\frac{5}{2}}}\Big( \fint_{\PP_{r_i(z_i)}}|\d-\d_{z_i, r_i}|^{6} \Big)^{\frac{1}{2}}.
    \end{split}
  \end{equation*}
  From \eqref{3.7}, we see that 
  \begin{align*}
    C_i r_i^{3}&\le 2\Phi(z_i, \tau r_i)\le 2\max\left\{ \tau^{-4}, 2^{3}\tau^{-{\frac{5}{2}}} \right\}\Phi(z_i, r_i)\\
    &=2\max\left\{ \tau^{-4}, 2^{3}\tau^{-{\frac{5}{2}}} \right\}\varepsilon_i^{3}
  \end{align*}
  so that 
  \begin{equation}
    r_i\le \Big( \frac{\varepsilon_i^{3}}{2C_i \max\big\{ \tau^{-4}, 2^{3}\tau^{-{\frac{5}{2}}} \big\} } \Big)^{\frac{1}{3}}\to 0. 
    \label{eqn:riestimate}
  \end{equation}
  Define the blowing-up sequence 
  \begin{equation*}
    (\u_i, \d_i, P_i):=\left( r_i \u, \d, r_i^2 P \right)(x_i+r_i x, t_i+r_i^2 t), \forall x\in {\R^3}, t>-1, 
  \end{equation*}
  and 
  \begin{equation*}
    (\hu_i, \hd_i, \hP_i)(z):=\left( \frac{\u_i}{\varepsilon_i}, \frac{\d_i-\overline{\d_i}}{\varepsilon_i}, \frac{P_i}{\varepsilon_i}  \right)(z), \forall z(x, t)\in \PP_1(0), 
  \end{equation*}
  where 
  \begin{equation*}
    \overline{\d}_i=\fint_{\PP_1(0)}\d_i dxdt.
  \end{equation*}
  Then $(\hu_i, \hd_i, \hP_i)$ satisfies
  \begin{equation}
    \left\{
    \begin{array}{l}
      \fint_{\PP_1(0)}\hd_i dxdt=0,\quad  |\od_i|=|\d_{z_i, r_i}|\le M_0,  \\
      \int_{\PP_1(0)}\left( |\hu_i|^{3}+|\nabla\hd_i|^{3} \right)dxdt+\left( \int_{\PP_1(0)}|\hP_i|^{\frac{3}{2}}dxdt \right)^{2}\\
      +\left(\fint_{\PP_1(0)}|\hd_i|^{6} dxdt\right)^{\frac{1}{2}}=1,  \\
      \tau^{-2}\int_{\PP_\tau(0)}\left( |\hu_i|^{3}+|\nabla\hd_i|^{3} \right)dxdt+\left( \tau^{-2}\int_{\PP_\tau(0)}|\hP_i|^{\frac{3}{2}}dxdt \right)^{2} 
       \\+\left(\fint_{\PP_\tau(0)}|\hd_i-(\hd_i)_{0,\tau}|^{6} dxdt\right)^{\frac{1}{2}}
      \ge \frac{1}{2}\max\Big\{ 1, C_i\big(\frac{r_i}{\varepsilon_i} \big)^{3}\Big\}.
    \end{array}
    \right.
    \label{3.8}
  \end{equation}
It follows from \eqref{4.4}, \eqref{4.5} that
  \begin{equation}
    \left\{
    \begin{array}{l}
      \fint_{\PP_1(0)}|\d_i|^{6}dxdt \le C\big( \fint_{\PP_1(0)}|\d_i-\overline{\d_i}|^{6}dxdt+|\overline{\d_i}|^{6} \big)\le C\big( \varepsilon_i^{6}+M_0^{6} \big),  \\ 
      \fint_{\PP_1(0)}F(\d_i)^{\frac{3}{2}}dxdt\le C\fint_{\PP_1(0)}||\d_i|^2-1|^{3}dxdt
      \le C\big( \varepsilon_i^{6}+M_0^{6}+1 \big), \\
      \fint_{\PP_1(0)}|\f(\d_i)|^{2} dxdt\le C\left(\fint_{\PP_1(0)}|\d_i|^{6} dxdt+1\right)\le C(\varepsilon_i^{6}+M_0^{6}+1),\\
      \fint_{\PP_1(0)}|\pa_{\d}\f(\d_i)|^{3} dxdt\le C\left(\fint_{\PP_1(0)}|\d_i|^{6} dxdt+1\right)\le C(\varepsilon_i^{6}+M_0^{6}+1).
    \end{array}
    \right.
    \label{eqn:Fdiestimate}
  \end{equation}
Furthermore, $(\hu_i, \hd_i, \hP_i)$ is a suitable weak solution of the blowing-up version of \eqref{eqn:ELmodel1}:
 \begin{equation}
   \left\{
   \begin{array}{l}
     \pa_t \hu_i+\varepsilon_i \hu_i\cdot \nabla\hu_i+\nabla\hP_i\\
     =\Delta\hu_i-\varepsilon_i \nabla\hd_i\cdot \Delta \hd_i+\frac{r_i^2}{\varepsilon_i}\nabla\d_i\cdot \f(\d_i)
      -\nabla\cdot S_\alpha[\Delta\hd_i-\frac{r_i^2}{\varepsilon_i}\f(\d_i),\d_i], \\
     \dv \hu_i=0, \\
     \pa_t\hd_i+\varepsilon_i \hu_i\cdot \nabla\hd_i-T_\alpha[\nabla\hu_i,\d_i]=\Delta\hd_i-\frac{r_i^2}{\varepsilon_i}\f(\d_i). 
   \end{array}
   \right.
   \label{eqn:scaledELsyst}
 \end{equation}
 From \eqref{3.8}, we assume that there exists
 \begin{equation}
   (\hu, \hd, \hP)\in L^{3}(\PP_1(0))\times L_t^{3} W^{1, 3}_x(\PP_1(0))\times L^{\frac{3}{2}}(\PP_1(0))
 \label{}
 \end{equation}
such that, after passing to a subsequence, 
\begin{equation*}
  (\hu_i, \hd_i, \hP_i)\rightharpoonup (\hu, \hd, \hP) \text{ in }L^{3}(\PP_1(0))\times L_t^{3} W_x^{1, 3}(\PP_1(0))\times L^{\frac{3}{2}}(\PP_1(0)).
\end{equation*}
It follows from \eqref{3.8} and the lower semicontinuity that 
\begin{equation}
  \int_{\PP_1(0)}\big( |\hu|^{3}+|\nabla\hd|^{3} \big)dxdt+\left( \int_{\PP_1(0)}|\hP|^{\frac{3}{2}}dxdt \right)^{2}+\left(\fint_{\PP_1(0)}|\hd|^{6} dxdt\right)^{\frac{1}{2}}\le 1.
  \label{eqn:blowupbound}
\end{equation}
We claim that 
\begin{equation}
  \left\|\hu_i\right\|_{L_t^\infty L_x^2\cap L_t^2 H_x^1(\PP_{\frac{1}{2}}(0))}+\left\|\hd_i\right\|_{L_t^\infty H_x^1\cap L_t^2 H_x^2(\PP_{\frac{1}{2}}(0))}\le C<\infty.
  \label{eqn:huhdbound}
\end{equation}
In fact, if we choose a cut-off function $\p \in C_0^\infty(\PP_1(0))$ such that 
\begin{equation*}
  0\le \p\le 1, \p \equiv 1 \text{ on }\PP_{\frac{1}{2}}(0), \text{ and }|\pa_t \p|+|\nabla\p|+|\nabla^2\p|\le C.
\end{equation*}
Define
\begin{equation*}
  \p_i(x, t):=\p\left( \frac{x-x_i}{r_i}, \frac{t-t_i}{r_i^2} \right), \forall (x, t)\in {\R^3}\times(0, \infty). 
\end{equation*}
Replacing $\p$ by $\p_i^2$ in \eqref{eqn:localeng}, by Young's inequality we can show
\begin{equation}
  \begin{split}
    &\sup_{t_i-\frac{r_i^2}{4}\le t\le t_i}\int_{B_{r_i}(x_i)}\left( |\u|^2+|\nabla\d|^2
     +F(\d)\right)\p_i^2 dx\\
     &+\int_{\PP_{r_i}(z_i)}\left( |\nabla\u|^2+|\Delta \d|^2+|\f(\d)|^2 \right)\p_i^2 dxdt\\
    &\le C\Bigg[\int_{\PP_{r_i}(z_i)}(|\u|^2+|\nabla\d|^2)|(\pa_t+\Delta)\p_i^2|+F(\d)|\pa_t\p_i^2| dxdt\\
    &\quad +\int_{\PP_{r_i}(z_i)}(|\u|^2+|\nabla\d|^2+|P|)|\u||\nabla\p_i^2|dxdt\\
    &+\int_{\PP_{ r_i}(z_i)}|\d|^2|\u|^2|\nabla\p_i|^2+|\d|^2|\nabla\d|^2|\nabla\p_i|^2 dxdt\\
    &\quad +\int_{\PP_{r_i}(z_i)}|\nabla\d|^2 (|\nabla^2(\p_i^2)|+|\nabla\p_i|^2)+|\pa_\d \f(\d)||\nabla\d|^2 \p_i^2
    dxdt\Bigg].
  \end{split}
  \label{4.14}
\end{equation}
By rescaling and using the estimates   \eqref{eqn:riestimate}, \eqref{3.8}, and \eqref{eqn:Fdiestimate},
we can show that  
\begin{equation}
  \begin{split}
    &\sup_{-\frac{1}{4}\le t\le 0}\int_{B_{\frac{1}{2}}(0)}\left( |\hu_i|^2+|\nabla\hd_i|^2 \right)dx+\int_{\PP_{\frac{1}{2}(0)}}\left(|\nabla\hu_i|^2+|\nabla^2 \hd_i|^2\right)dxdt\\
 & \le C\int_{\PP_1(0)}\left[ \left( |\hu_i|^2+|\nabla\hd_i|^2 \right)+\frac{r_i^2}{\varepsilon_i^2} F(\d_i) \right] dxdt\\
  &\quad +C\varepsilon_i\int_{\PP_1(0)}(|\hu_i|^2+|\nabla\hd_i|^2+|\hP_i|)|\hu_i| dxdt\\
  &\quad +C\int_{\PP_1(0)}(|\d_i|^2|\hu_i|^2 +|\d_i|^2|\nabla\hd_i|^2) dxdt\\
  &\quad +C\int_{\PP_1(0)}|\nabla\hd_i|^2+r_i^2|\nabla\hd_i|^2 |\pa_\d \f(\d_i)| dxdt\\
  &\le C.
  \end{split} 
  \label{}
\end{equation}
This yields \eqref{eqn:huhdbound}. Hence we may assume that 
\begin{equation}
  (\hu_i, \hd_i)\rightharpoonup (\hu, \hd) \text{ in }L_t^2 H_x^1(\PP_{\frac{1}{2}}(0))\times L_t^2 H_x^2(\PP_{\frac{1}{2}}(0)).
  \label{eqn:udweakconv}
\end{equation}
From $\frac{r_i}{\varepsilon_i}\to 0$ and $\big|\fint_{\PP_1(0)}\d_idxdt\big|\le M_0$, we have
\begin{align*}
  \big|\fint_{\PP_{\frac{1}{2}}(0)}\d_i dxdt\big|&\le \big|\fint_{\PP_{\frac{1}{2}}(0)}\big(\d_i-\fint_{\PP_1(0)}\d_i\big)dxdt\big|+\big|\fint_{\PP_1(0)}\d_idxdt\big|\\
  &\le C \big(\fint_{\PP_1(0)}|\d_i-\od_i|^{6}dxdt\big)^{\frac{1}{6}}+M_0
  \le C\varepsilon_i+M_0\le C.
\end{align*}
Thus by the same interpolation as in \eqref{eqn:L10interpolation}, we have
\begin{align*}
	\|\d_i\|_{L^{10}(\PP_{\frac{1}{2}}(0))}&\le C, \\
	\fint_{\PP_{\frac{1}{2}}(0)}|\f(\d_i)|^{\frac{10}{3}} dxdt&\le C,\\ \fint_{\PP_{\frac{1}{2}}(0)}|\d_i\otimes\f(\d_i)|^{\frac{5}{2}} dxdt &\le C\\
\fint_{\PP_{\frac{1}{2}}(0)}	F(\d_i)^{\frac{5}{2}} dxdt&\le C, 
\end{align*}
 and there exists a constant $\od\in {\R^3}$, with $|\od|\le M_0$, such that, after passing to subsequence,  $\od_i\to\od$, 
\begin{equation*}
  \d_i\to \od \text{ in }L^{6}(\PP_{\frac{1}{2}}(0)),
\end{equation*}
and 
\begin{equation}
\begin{split}
  \frac{r_i^2}{\varepsilon_i}\f(\d_i)\to 0 \text{ in } L^{\frac{10}{3}}(\PP_{\frac{1}{2}}(0)),\\
  \frac{r_i^2}{\varepsilon_i}\f(\d_i)\otimes\d_i\to 0 \text{ in }L^{\frac{5}{2}}(\PP_{\frac{1}{2}(0)}),\\ \frac{r_i^2}{\varepsilon_i}\d_i\otimes\f(\d_i)\to 0 \text{ in } L^{\frac{5}{2}}(\PP_{\frac{1}{2}(0)}),\\
  \frac{r_i^2}{\varepsilon_i^2} F(\d_i)\to 0 \text{ in } L^{\frac{5}{2}}(\PP_{\frac{1}{2}(0)}).
  \end{split}
  \label{eqn:fdL2conv}
\end{equation}
Hence $(\hu, \hd, \hP):\PP_{\frac{1}{2}}(0)\to {\R^3}\times{\R^3}\times \R$ solves the linear system:
\begin{equation}
  \left\{
  \begin{array}{l}
    \pa_t\hu+\nabla\hP-\Delta\hu=-\nabla\cdot S_\alpha[\Delta\hd,\od], \\
    \dv \hu=0, \\
    \pa_t\hd-\Delta\hd=T_\alpha[\nabla\hu,\od].
      \end{array}
  \right.
  \label{eqn:blowupeq}
\end{equation}
By Lemma \ref{lemma:blowupregularity} and \eqref{eqn:blowupbound}, we have that 
$
  (\hu, \hd)\in C^\infty(\PP_{\frac{1}{4}}), \hP\in L^\infty\big(-\big[ \frac{1}{16}, 0 \big], C^\infty(B_{\frac{1}{4}}(0))\big)
$  
satisfies
\begin{equation}
  \begin{split}
    &\tau^{-2}\int_{\PP_\tau(0)}\left( |\hu|^{3}+|\nabla\hd|^{3} \right)dxdt+\big( \tau^{-2}\int_{\PP_\tau(0)}|\hP|^{\frac{3}{2}}dxdt \big)^{2}\\
    &\le C\tau^{3}\big[\int_{\PP_{\frac{1}{2}}(0)}\big( |\hu|^{3}+|\nabla\hd|^{3} \big)dxdt+\big( \int_{\PP_{\frac{1}{2}}(0)}|\hP|^{\frac{3}{2}}dxdt \big)^2\big]\\
    &\le C\tau^{3}, \quad\forall \tau\in \big( 0, \frac{1}{8} \big).
  \end{split}
  \label{}
\end{equation}
and $\exists\alpha_0\in(0,1)$ such that
\begin{equation}
  \big( \fint_{\PP_\tau(0)}|\hd-\hd_{0,\tau}|^{6} dxdt \big)^{\frac{1}{2}}
  \le C\big( \fint_{\PP_{\frac{1}{2}}(0)}|\hd|^{6} dxdt\big)^{\frac{1}{2}} \tau^{3\alpha_0}
  \le C \tau^{3\alpha_0}, \quad\forall \tau\in \big( 0, \frac{1}{8} \big).
  \label{}
\end{equation}
We now claim that 
\begin{equation}
\begin{array}{ll}
 (\hu_i, \nabla\hd_i)\to (\hu, \nabla\hd) &\text{ in }L^{3}(\PP_{\frac{3}{8}}(0)),\\
 \hd_i\to \hd &\text{ in }L^{6}(\PP_{\frac{3}{8}}(0)).
\end{array}
\label{eqn:udstongconv}
\end{equation}
In fact, from the equation for $\hu_i$ and $\hd_i$ in \eqref{eqn:scaledELsyst} we can conclude that 
\begin{equation*}
  \left\|\pa_t\hu_i\right\|_{L_t^2H_x^{-1}+L_t^{\frac{6}{5}}L_x^{\frac{6}{5}}+L_t^{\frac{3}{2}}W_x^{-1, \frac{3}{2}}(\PP_{\frac{3}{8}}(0))}\le C,
\end{equation*}
and
\begin{equation}
  \big\|\pa_t \hd_i\big\|_{L^{\frac{3}{2}}(\PP_{\frac{3}{8}}(0))}
  \le C. 
\end{equation}
Thus \eqref{eqn:udstongconv} follows from Aubin--Lions' compactness Lemma. 
This implies that for any $\tau\in(0,\frac{1}{8})$, 
\begin{eqnarray}
  \tau^{-2}\int_{\PP_\tau(0)}\left( |\hu_i|^{3}+|\nabla\hd_i|^{3} \right)dxdt
  &=&\tau^{-2}\int_{\PP_\tau(0)}(|\hu|^{3}+|\nabla\hd|^{3})dxdt+\tau^{-2}o(1)\nonumber\\
  &\le& C\tau^{3}+\tau^{-2}o(1),\nonumber\\
  \big( \fint_{\PP_\tau(0)}|\hd_i-(\hd_i)_{0, \tau}|^{6} dxdt \big)^{\frac{1}{2}}
  &\le& C \tau^{3\alpha_0}+o(1),
  \label{eqn:huhdestimate}
\end{eqnarray}
where $\lim_{i\to \infty}o(1)=0$. 

Now we need to estimate the pressure $\hP_i$. By taking divergence of the $\hu_i$ equation in \eqref{eqn:blowupeq} we see that 
\begin{eqnarray}
  -\Delta\hP_i&=&\varepsilon_i\dv^2\big[ \hu_i\otimes\hu_i+\nabla\hd_i\odot\nabla\hd_i-\big(\frac{1}{2}|\nabla\hd_i|^2 +\frac{r_i^2}{\varepsilon_i^2} F(\d_i) \big)I_3\big]\nonumber\\
  &&+\dv^2 S_\alpha[\Delta\hd_i-\frac{r_i^2}{\varepsilon_i}\f(\d_i),\d_i]   \ \text{ in }\ B_1.
\end{eqnarray}
We claim that 
\begin{equation}
  \tau^{-2}\int_{\PP_\tau(0)}|\hP_i|^{\frac{3}{2}}dxdt\le C\tau+ C\tau^{-2} (\varepsilon_i+o(1)). 
  \label{eqn:hPiestimate}
\end{equation}
Since $S_\alpha[\Delta\hd_i, \d_i]$ does not necessarily have a small $L^2$-norm in $\PP_{\frac{1}{2}}(0)$,
to achieve \eqref{eqn:hPiestimate} we will show the following strong convergence in $L^2$:
\begin{equation}
  (\nabla\hu_i, \Delta\hd_i)\to (\nabla\hu, \Delta\hd) \text{ in }L^2\big( \PP_{\frac{3}{8}}(0) \big).
  \label{eqn:stongL2conv}
\end{equation}
In order to prove \eqref{eqn:stongL2conv}, first observe that by subtracting the equation \eqref{eqn:scaledELsyst} from the equations \eqref{eqn:blowupeq}, we see that 
\begin{equation*}
  (\tu_i, \td_i, \tP_i):=\left( \hu_i-\hu, \hd_i-\hd, \hP_i-\hP \right)
\end{equation*}
solves the following system of equations in $\PP_{\frac{1}{2}}(0)$:
\begin{equation}
  \left\{
  \begin{array}{l}
    \pa_t\tu_i-\Delta\tu_i+\nabla\tP_i=-\varepsilon_i\hu_i\cdot\nabla\hu_i-\varepsilon_i \nabla\hd_i\cdot \Delta \hd_i+\frac{r_i^2}{\varepsilon_i}\nabla\d_i\cdot \f(\d_i)\\
-\nabla\cdot S_\alpha[\Delta\hd_i-\frac{r_i^2}{\varepsilon_i}\f(\d_i),\d_i]
   +\nabla\cdot S_\alpha[\Delta\hd,\od], \\
    \dv \tu_i=0, \\
    \pa_t\td_i-\Delta\td_i=-\varepsilon_i\hu_i\cdot\nabla\hd_i-\frac{r_i^2}{\varepsilon_i}\f(\d_i)+T_\alpha[\nabla\hu_i,\d_i] -T_\alpha[\nabla\hu,\od].
  \end{array}
  \right.
  \label{eqn:differenceequ}
\end{equation}
Since $(\hu_i, \hd_i, \hP_i)$ is a suitable weak solution of \eqref{eqn:scaledELsyst} and 
Lemma 4.2 guarantees the smoothness of $(\hu,\hd, \hP)$, it is not hard to see
that \eqref{eqn:differenceequ} also enjoys a local energy inequality which leads to \eqref{eqn:stongL2conv}. 
In fact,  multiplying the $\tu_i$ equation by $\tu_i\p$, and $\nabla\td_i$ equation by $\nabla\td_i \p$, integrating the resulting equation over ${\R^3}\times [0,T]$, and applying the integration by parts, we obtain that
\begin{eqnarray}
 && \int_{\R^3}|\tu_i|^2\p(x, t) dx+2\int_{0}^{t}\int_{\R^3}|\nabla\tu_i|^2\p dxds\nonumber\\
  &&\le\int_{0}^{t}\int_{\R^3}|\tu_i|^2(\pa_t\p+\Delta\p)dxds\nonumber\\
  &&+\int_{0}^{t}\int_{\R^3} [\varepsilon_i |\hu_i|^2\hu_i\cdot \nabla\p+2 \varepsilon_i (\hu_i\cdot \nabla\hu_i)\cdot \hu \p+2\tP_i\tu_i\cdot \nabla\p] dxds\nonumber\\
 && +2\int_{0}^{t}\int_{\R^3}(-\varepsilon_i \nabla\hd_i\cdot \Delta\hd_i\cdot(\hu_i-\hu)\p+\frac{r_i^2}{\varepsilon_i}\nabla\d_i\cdot \f(\d_i)\cdot \tu_i\p) dxds\nonumber\\
 && -\frac{2r_i^2}{\varepsilon_i}\int_{0}^{t}\int_{\R^3}S_\alpha [\f(\d_i), \d_i]:(\nabla\tu_i \p+\tu_i\otimes\nabla\p) dxds\nonumber\\
 && +2\int_{0}^{t}\int_{\R^3} \left( S_\alpha[\Delta\hd_i, \d_i]-S_\alpha[\Delta\hd,\od] \right)
  :\left( \nabla\hu_i \p-\nabla\hu \p +\tu_i\otimes \nabla\p \right) dxds,
  \label{eqn:tuenergy}
\end{eqnarray}
and
\begin{eqnarray}
 && \int_{\R^3}|\nabla\hd_i|^2 \p(x, t)dx+2\int_{0}^{t}\int_{\R^3}|\Delta\td_i|^2 \p dxds\nonumber\\
 && \le\int_{0}^{t}\int_{\R^3}(|\nabla\td_i|^2 (\pa_t\p+\Delta\p)+2\varepsilon_i\hu_i \cdot \nabla\hd_i\cdot (\Delta\hd_i\p-\Delta\hd\p+\nabla\p\cdot\nabla\td_i )) dxds\nonumber\\
 && +\frac{2r_i^2}{\varepsilon_i}\int_{0}^{t}\int_{\R^3}  \f(\d_i)\cdot (\Delta\td_i\p+\nabla\p\cdot\nabla\td_i ) dxds\nonumber\\
&&  -2\int_{0}^{t}\int_{\R^3} \left[ T_\alpha[\nabla\hu_i,\d_i]-T_\alpha[\nabla\hu,\od] \right]
  \cdot\big( \Delta\hd_i\p-\Delta\hd \p+\nabla\p\cdot\nabla\td_i \big) dxds.
  \label{eqn:tdenergy}
\end{eqnarray}
Recall that 
\begin{align}
 \int_{0}^{t}\int_{\R^3}S_\alpha[\Delta\hd_i,\d_i]:\nabla\hu_i\p dxdt&= \int_{0}^{t}\int_{\R^3}T_\alpha[\nabla\hu_i,\d_i]\cdot \Delta\hd_i\p dxdt, \\
 \int_{0}^{t}\int_{\R^3} S_\alpha [\Delta\hd,\od]:\nabla\hu\p dxdt&=\int_{0}^{t}\int_{\R^3}T_\alpha[\nabla\u,\od]:\Delta\hd\p dxdt.
  \label{eqn:cancellation}
\end{align}
Therefore we can add \eqref{eqn:tuenergy} and \eqref{eqn:tdenergy} to obtain that
\begin{equation}
  \begin{split}
    &\int_{\R^3}(|\tu_i|^2+|\nabla\td_i|^2)\p(x,t)dx+2\int_{0}^{t}\int_{\R^3}
    \left( |\nabla\tu_i|^2+|\Delta\td_i|^2 \right)\p dxds\\
    &\le\int_{0}^{t}\int_{\R^3}[(|\tu_i|^2+|\nabla\td_i|^2)(\pa_t\p+\Delta\p)\\
    &\qquad+(\varepsilon_i |\hu_i|^2 \hu_i+2\tP_i\tu_i)\cdot \nabla\p+2\varepsilon_i(\hu_i\cdot \nabla\hu_i)\cdot \hu\p]dxds\\
    &+2\int_{0}^{t}\int_{\R^3} \varepsilon_i \nabla\hd_i\cdot \Delta\hd_i\cdot\hu \p-\varepsilon_i \hu_i\cdot \nabla\hd_i\cdot \Delta\hd \p+2\varepsilon_i \hu_i\cdot \nabla\hd_i\cdot (\nabla\p\cdot \nabla\td_i) dxds\\
    &+\frac{2r_i^2}{\varepsilon_i}\int_{0}^{t}\int_{\R^3}\nabla\d_i\cdot \f(\d_i)\cdot \tu_i\p+\f(\d_i)\cdot \left( \Delta\td_i\p+\nabla\p\cdot \nabla\td_i \right)dxds\\
    &-\frac{2r_i^2}{\varepsilon_i^2}\int_{0}^{t}\int_{\R^3}  S_\alpha[\f(\d_i),\d_i]:(\nabla\tu_i\p+\tu_i\otimes\nabla\p) dxds\\
    &+2\int_{0}^{t}\int_{\R^3} S_\alpha[\Delta\hd_i,\d_i] :(\tu_i\otimes\nabla\p-\nabla\hu\p) dxds\\
    &-2\int_{0}^{t}\int_{\R^3} T_\alpha[\nabla\hu_i,\d_i]\cdot (\nabla\p\cdot \nabla\td_i-\Delta\hd\p) dxds\\
    &-2\int_{0}^{t}\int_{\R^3} S_\alpha[\Delta\hd,\od]:(\nabla\hu_i\p+\tu_i\otimes\nabla\p)dxds\\
    &+2\int_{0}^{t}\int_{\R^3} T_\alpha[\nabla\hu,\od]\cdot (\Delta\hd_i\p +\nabla\p\cdot \nabla\td_i)dxds\\
    &:=\sum_{k=1}^{8}I_k(i).
  \end{split}
    \label{eqn:keycancel}
\end{equation}
From the  convergence \eqref{eqn:udweakconv}, we know that 
\begin{equation*}
  \begin{split}
  &\lim_{i\to\infty}\left\|(\tu_i, \nabla\td_i)\right\|_{L^{3}(\PP_{\frac{3}{8}}(0))}=0, \\
  &\tP_i\rightharpoonup 0 \text{ in }L^{\frac{3}{2}}(\PP_{\frac{3}{8}}(0)), \\
  &(\nabla\tu_i, \nabla^2\td_i)\rightharpoonup (0, 0) \text{ in }L^2(\PP_{\frac{3}{8}}(0)).
  \end{split}
\end{equation*}
This, together with \eqref{eqn:fdL2conv}, implies that as $i\to\infty$, 
$\sum_{k=1}^{4}I_k\to 0$ and
\begin{equation}
  \begin{split}
    &I_5\to- 2\int_{0}^{t}\int_{\R^3} S_\alpha[\Delta\hd,\od]:\nabla\hu \p dxds\\
    &I_6\to 2\int_{0}^{t}\int_{\R^3}  T_\alpha[\nabla\hu,\od]\cdot \Delta\hd \p dxds\\
    &I_7\to -2\int_{0}^{t}\int_{\R^3} S_\alpha[\Delta\hd,\od]:\nabla \hu \p dxds\\
    &I_8\to 2 \int_{0}^{t}\int_{\R^3}  T_\alpha[\nabla\hu,\od]:\Delta 
    \hd\p dxds,
  \end{split}
  \label{}
\end{equation}
Therefore
\begin{equation*}
    \sum_{k=1}^{8}I_k(i)
    \to 4\int_{0}^{t}\int_{\R^3} T_\alpha[\nabla\hu,\od]:\Delta\hd\p - S_\alpha[\Delta\hd,\od]:\nabla\hu\p dxds=0,
\end{equation*}
and \eqref{eqn:stongL2conv} holds.

Let $\eta\in C_0^\infty(B_{\frac{3}{8}}(0))$ be such that $\eta\equiv 1$ in $B_{\frac{5}{16}}(0)$, $0\le \eta\le 1$. For any 
$-(\frac{3}{8})^2\le t\le 0$, define $\hP_i^{(1)}(\cdot , t):{\R^3}\to \R$ by
\begin{multline}
  \hP_i^{(1)}(x, t)=\int_{\R^3}\nabla_x^2 G(x-y)\Big\{ \varepsilon_i\eta\big[ \hu_i\otimes\hu_i+\nabla\hd_i\odot\nabla\hd_i-\big(\frac{1}{2}|\nabla\hd_i|^2 +\frac{r_i^2}{\varepsilon_i^2} F(\d_i) \big)I_3 \big]\\
  -\frac{r_i^2}{\varepsilon_i}[S_\alpha[\f(\d_i),\d_i]]
  +\big[  S_\alpha[\Delta\hd_i,\d_i]-S_\alpha[\Delta\hd,\od]\big]\Big\}(y, t)dy, 
  \label{}
\end{multline}
and $\hP_i^{(2)}(\cdot , t)=(\hP_i-\hP_i^{(1)})(\cdot , t)$. Then  
\begin{equation}
  -\Delta\hP_i^{(2)}=\dv^2S_\alpha[\Delta\hd,\od]\ \ \text{ in }\ \ B_{\frac{5}{16}}(0).
  \label{}
\end{equation}
For $\hP_i^{(1)}$, by the Calderon-Zgymund theory we have that 
\begin{equation}
\begin{split}
  \big\|\hP_i^{(1)}\big\|_{L^{\frac{3}{2}}({\R^3})}&\le C  \Big[\varepsilon_i\big( \big\|\hu_i\big\|_{L^{3}(B_{\frac{3}{8}}(0))}^2+\big\|\nabla\hd_i\big\|_{L^{3}(B_{\frac{3}{8}}(0))}^2+\frac{r_i^2}{\varepsilon_i^2}\big\|F(\d_i)\big\|_{L^{\frac{3}{2}}(B_{\frac{3}{8}}(0))}\big)\\
  &+ \frac{r_i^2}{\varepsilon_i}\big\||\f(\d_i)||\d_i|\big\|_{L^{\frac{3}{2}}(B_{\frac{3}{8}}(0))}+\big\|S_\alpha[\Delta\hd_i,\d_i]-
  S_\alpha[\Delta\hd,\od]\big\|_{L^{\frac{3}{2}}(B_{\frac{3}{8}}(0))}
   \Big]\\
    &\le C\Big[\varepsilon_i\big( \big\|\hu_i\big\|_{L^{3}(B_{\frac{3}{8}}(0))}^2+\big\|\nabla\hd_i\big\|^2_{L^{3}(B_{\frac{3}{8}}(0))}+\frac{r_i^2}{\varepsilon_i^2}\big\|F(\d_i)\big\|_{L^{\frac{3}{2}}(B_{\frac{3}{8}}(0))} \big)\\
   &+\frac{r_i^2}{\varepsilon_i}\big\||\f(\d_i)||\d_i|\big\|_{L^{\frac{3}{2}}(B_{\frac{3}{8}}(0))}+
   \big\|\d_i\big\|_{L^{6}(B_{\frac{3}{8}}(0))}\big\|\Delta\hd_i-\Delta\hd\big\|_{L^2(B_{\frac{3}{8}}(0))}\\
   &+\big\|\d_i-\od\big\|_{L^{6}(B_{\frac{3}{8}}(0))}\big\|\Delta\hd\big\|_{L^2(B_{\frac{3}{8}}(0))}\Big]
  \end{split}
\end{equation}
Hence we have
\begin{equation}
  \begin{split}
    \big\|\hP^{(1)}_i\big\|_{L^{\frac{3}{2}}(\PP_{\frac{1}{3}}(0))}\le C(\varepsilon_i+o(1)). 
  \end{split}
  \label{4.24}
\end{equation}
From the standard theory on linear elliptic equations, $\hP_i^{(2)}\in C^\infty(B_{\frac{5}{16}}(0))$ satisfies that for any $0<\tau<\frac{9}{32}$, 
\begin{equation}
  \begin{split}
    \tau^{-2}\int_{\PP_\tau(0)}|\hP_i^{(2)}|^{\frac{3}{2}} dxdt&\le C\tau\Big[ \int_{\PP_{\frac{9}{32}}(0)}|\hP_i^{(2)}|^{\frac{3}{2}} dxdt+\big\|\nabla^3\hd\big\|_{L^\infty(\PP_{\frac{9}{32}(0)})}^{\frac{3}{2}} \Big]\\
    &\le C\tau\Big[ \int_{\PP_{\frac{9}{32}}(0)}(|\hP_i|^{\frac{3}{2}}+|\hP_i^{(1)}|^{\frac{3}{2}})dxdt+\big\|\nabla^3 \hd\big\|_{L^\infty(\PP_{\frac{9}{32}}(0))}^{\frac{3}{2}} \Big]\\
    &\le C\tau(1+\varepsilon_i+o(1)). \\
  \end{split}
  \label{4.25}
\end{equation}
Combining \eqref{4.24} with \eqref{4.25} yields \eqref{eqn:hPiestimate}. 
It follows from \eqref{eqn:huhdestimate} and \eqref{eqn:hPiestimate} that there exist sufficiently small $\tau_0\in(0,\frac{1}{4})$ and sufficiently large $i_0$, depending on $\tau_0$, such that for any $i\ge i_0$, it holds that 
\begin{eqnarray*}
  \tau_0^{-2}\int_{\PP_{\tau_0}(0)}(|\hu_i|^{3}+|\nabla\hd_i|^{3})dxdt
  &+&\big( \tau_0^{-2}\int_{\PP_{\tau_0}(0)}|\hP_i|^{\frac{3}{2}} dxdt \big)^{\frac{3}{2}}\\
  &+&\big( \fint_{\PP_{\tau_0}(0)}|\hd_i-(\hd_i)_{\tau_0, 0}|^{6}dxdt \big)^{\frac{1}{2}}
  \le \frac{1}{4}.
\end{eqnarray*}
This contradicts \eqref{3.8}. The proof of Lemma \ref{lemma:smallregularity} is completed.  
\end{proof}

Now we will establish the smoothness of the limit equation \eqref{eqn:blowupeq} in the following lemma.
\begin{lemma}
  Assume that $(\hu, \hd)\in (L_t^\infty L_x^2\cap L_t^2 H_x^1)(\PP_{\frac{1}{2}}(0))\times\left( L_t^\infty H_x^1\cap L_t^2 H_x^2 \right)(\PP_{\frac{1}{2}})$ and $\hP\in L^{\frac{3}{2}}(\PP_{\frac{1}{2}}(0))$ is a weak solution of the linear system \eqref{eqn:blowupeq}, then $(\hu, \hd)\in C^\infty(\PP_{\frac{1}{4}}(0))$, and the following estimate
  \begin{equation}
    \tau^{-2}\int_{\PP_\tau(0)}\left( |\hu|^{3}+|\nabla\hd|^{3}+|\hP|^{\frac{3}{2}} \right) dxdt\le C\tau^3\int_{\PP_{\frac{1}{2}}(0)}\left( |\hu|^{3}+|\nabla\hd|^{3}+|\hP|^{\frac{3}{2}} \right) dxdt
    \label{eqn:smoothestimate}
  \end{equation}
  holds for any $\tau\in ( 0, \frac{1}{8})$. 
  \label{lemma:blowupregularity}
\end{lemma}
\begin{proof}
  The smoothness of the limit equation \eqref{eqn:blowupeq} doesn't follow from the standard theory of linear equations, since the source term of $\hu$ equations involve terms depending on the third order derivatives of $\hd$. It is based on higher order energy methods, for which the cancellation property, as in the derivation of local energy inequality for suitable weak solution to \eqref{eqn:ELmodel1}, plays a critical role. This strategy has been adapted by Huang--Lin--Wang
  in\cite[Lemma 3.2]{HuangLinWang2014EricksenLeslieTwoDimension}  for the full Ericksen--Leslie system in 2D. However, it is more delicate here due to the low temporal integrability of pressure. To address this issue, we split the pressure into two parts $\hP^{(1)}$ and $\hP^{(2)}$, where $\hP^{(1)}$ solves the Poisson equation involving $\Delta \hd$ which belongs to $L^2$, and $\hP^{(2)}$, while is only $L^\frac{3}{2}$ in time, is harmonic in space. In fact, if we take the
  divergence of the equation \eqref{eqn:blowupeq}$_1$, then we have $\hP$ satisfies the following Poisson equation:
  \begin{equation}
    -\Delta\hP=\dv^2S_\alpha[\Delta\hd, \od] \text{ in }\PP_{\frac{1}{2}}. 
    \label{eqn:hPeq}
  \end{equation}
  Now let $\zeta\in C_0^\infty(B_{\frac{1}{2}}(0))$ be a cut-off function of $B_{\frac{3}{8}}(0)$, i.e., $\zeta\equiv 1$ on $B_{\frac{3}{8}}(0)$, $0\le \zeta\le 1$. Define $\hP^{(1)}(\cdot , t):\R^3\to \R$, 
  \begin{equation*}
    \hP^{(1)}(x,t):=\int_{\R^3}\nabla_x^2 G(x-y)\zeta(y) S_\alpha[\Delta\hd, \od](y,t)dy, 
  \end{equation*}
  and $\hP^{(2)}(\cdot , t):=(\hP-\hP^{(1)})(\cdot , t)$. For $\hP^{(1)}$, by Calderon-Zygmund's singular integral estimate we have
  \begin{equation*}
    \left\|\hP^{(1)}(\cdot , t)\right\|_{L^2(\R^3)}\le C\left\|\Delta\hd(\cdot , t)\right\|_{L^2(B_{\frac{1}{2}})}, \qquad -\frac{1}{4}\le t\le 0.
  \end{equation*}
  Hence we can integrate the inequality above in time to get
  \begin{equation}
    \int_{\PP_{\frac{1}{2}}}|\hP^{(1)}|^2 dxdt\le C\int_{\PP_{\frac{1}{2}}}|\Delta\hd|^2 dxdt.
    \label{4.41}
  \end{equation}
  For $\hP^{(2)}$, it is easy to see that 
  \begin{equation}
    -\Delta\hP^{(2)}=0 \text{ in }B_{\frac{3}{8}}.
    \label{}
  \end{equation}
  By the standard regularity theory of harmonic function we have
  \begin{equation}
    \begin{split}
  \int_{\PP_{\frac{5}{16}}}|\nabla^l \hP^{(2)}|^{\frac{3}{2}} dxdt&
  \le C\int_{\PP_{\frac{3}{8}}}|\hP^{(2)}|^{\frac{3}{2}} dxdt\\
  &\le C \int_{\PP_{\frac{3}{8}}}(|\hP|^{\frac{3}{2}}+|\hP^{(1)}|^{\frac{3}{2}})dxdt\\
  &\le C\int_{\PP_\frac{1}{2}}|\hP|^{\frac{3}{2}} dxdt+  C\int_{\PP_{\frac{1}{2}}}|\hP^{(1)}|^2dxdt+C \\
  &\le C\int_{\PP_{\frac{1}{2}}}|\hP|^{\frac{3}{2}} dxdt+C \int_{\PP_{\frac{1}{2}}}|\Delta\hd|^2 dxdt+C , \qquad l=1,2. 
\end{split}\label{4.43}
\end{equation}
  Taking $\frac{\pa}{\pa x_i}$ of the linear equation \eqref{eqn:blowupeq} yields 
  \begin{equation}
    \left\{
    \begin{array}{l}
      \pa_t \hu_{x_i}+\nabla \hP_{x_i}-\Delta \hu_{x_i}=-\nabla\cdot S_\alpha[\Delta\hd, \od]_{x_i}, \\
      \nabla \cdot\hu_{x_i}=0, \\
      \pa \hd_{x_i}-\Delta\hd_{x_i}=T_{\alpha}[\nabla \hu, \od]_{x_i}.
    \end{array}
    \right.
    \label{eqn:blowupeqxi}
  \end{equation}
  For any $\eta\in C_0^\infty( B_{\frac{5}{16}})$, Multiplying the equation \eqref{eqn:blowupeqxi}$_1$ by $\hu_{x_i}\eta^2$ and the $\nabla\hd_{x_i}$  equation from \eqref{eqn:blowupeqxi}$_3$ by $\nabla\hd_{x_i}\eta^2$ and integrating the resulting equations over $B_{\frac{5}{16}}$, we obtain\footnote{Strictly speaking, we need to take finite quotient $D_h^j$ of \eqref{eqn:blowupeq} $(j=1,2,3)$ and then sending $h\to0$.}
  \begin{equation}
  \begin{split}
    &\frac{d}{dt}\int_{B_{\frac{5}{16}}}|\nabla\hu|^2\eta^2 dx+2\int_{B_{\frac{5}{16}}}|\nabla^2\hu|^2\eta^2 dx\\
    &=2\int_{B_{\frac{5}{16}}} [ \hP_{x_i} \hu_{x_i}\cdot\nabla(\eta^2)-
    \nabla \hu_{x_i}:\hu_{x_i}\otimes\nabla(\eta^2)]dx\\
    &\quad+2\int_{B_{\frac{5}{16}}}[S_{\alpha}[\Delta\hd, \od]_{x_i}:\hu_{x_i}\otimes\nabla(\eta^2)+S_\alpha[\Delta\hd, \od]_{x_i}:\nabla\hu_{x_i}\eta^2] dx.
    \label{4.45}
    \end{split}
  \end{equation}
\begin{equation}
\begin{split}
    &\frac{d}{dt}\int_{B_{\frac{5}{16}}}|\nabla^2\hd|^2\eta^2 dx+2\int_{B_{\frac{5}{16}}}|\Delta\nabla\hd|^2\eta^2 dx\\
    &=-2 \int_{B_{\frac{5}{16}}} \nabla_j \nabla\hd_{x_i}:\nabla\hd_{x_i}\otimes \nabla_j(\eta^2) dx\\
    &\ -2\int_{B_{\frac{5}{16}}}T_\alpha[\nabla\hu, \od]_{x_i}\cdot\nabla_j\hd_{x_i} \nabla_j(\eta^2)+T_{\alpha}[\nabla\hu, \od]_{x_i}\cdot\Delta \hd_{x_i}\eta^2 dx.
    \end{split}
    \label{4.46}
  \end{equation}
  Once again, we have the cancellation
  \begin{align*}
    &\int_{B_{\frac{5}{16}}}[S_\alpha[\Delta\hd, \od]_{x_i}:\nabla\hu_{x_i}\eta^2-T_\alpha[\nabla\hu, \od]_{x_i}\cdot \Delta\hd_{x_i}\eta^2]dx
    \\
    =&\int_{B_{\frac{5}{16}}}[S_{\alpha}[\Delta\hd_{x_i}, \od]:\nabla\hu_{x_i}\eta^2-T_{\alpha}[\nabla\hu_{x_i}, \od]\cdot \Delta\hd_{x_i}\eta^2]dx
    =0.
  \end{align*}
Now we add \eqref{4.45} and \eqref{4.46} together to get 
\begin{equation}
\begin{split}
  &\frac{d}{dt}\int_{B_{\frac{5}{16}}}\left( |\nabla\hu|^2+|\nabla^2\hd|^2 \right)\eta^2 dx
  +\int_{B_{\frac{5}{16}}}\left( |\nabla^2\hu|^2+|\Delta\nabla\hd|^2 \right)\eta^2 dx\\
  &=2\int_{B_{\frac{5}{16}}}\hP_{x_i}\hu_{x_i}\cdot\nabla(\eta^2)dx\\
  &\quad-2\int_{B_{\frac{5}{16}}} [\nabla\hu_{x_i}:\hu_{x_i}\otimes\nabla(\eta^2)+\nabla_j \nabla\hd_{x_i}:\nabla\hd_{x_i}\otimes\nabla_j(\eta^2)]dx\\
  &\quad+2\int_{B_{\frac{5}{16}}}[S_\alpha[\Delta\hd_{x_i}, \od]:\hu_{x_i}\otimes\nabla(\eta^2)-T_\alpha[\nabla\hu_{x_i}, \od]\cdot \nabla_j\hd_{x_i}\nabla_j(\eta^2)]dx\\
  &:=I_1+I_2+I_3.
  \label{4.47}
  \end{split}
\end{equation}
We have the following estimates:
\begin{equation*}
\begin{split}
  &|I_1|\le 2\Big|\int_{B_{\frac{5}{16}}}(\hP^{(1)}\hu_{x_ix_i}\cdot \nabla(\eta^2)+\hP^{(1)}\hu_{x_i}\cdot \nabla(\eta^2)_{x_i})dx\Big|+2\Big|\int_{B_{\frac{5}{16}}}\hu\cdot (\hP_{x_i}^{(2)}\nabla(\eta^2))_{x_i} dx\Big|\\
  &\quad\le\frac{1}{32}\int_{B_\frac{5}{16}}|\nabla^2\hu|^2\eta^2 dx+C\int_{B_{\frac{1}{2}}}(|\nabla\hu|^2\eta^2+|\nabla\hu|^2|\nabla\eta|^2)dx +C\int_{\supp \eta}|\hP^{(1)}|^2 dx\\
  &\qquad+ C\int_{\supp \eta}(|\hu|^{3}+|\nabla\hP^{(2)}|^{\frac{3}{2}}+|\nabla^2\hP^{(2)}|^{\frac{3}{2}}) dx. 
\end{split}
\end{equation*}
\begin{align*}
  |I_2|&\le \frac{1}{16}\int_{B_{\frac{5}{16}}}\left( |\nabla^2\hu|^2+|\Delta\nabla\hd|^2 \right)\eta^2 dx
  +C\int_{B_{\frac{5}{16}}}\left( |\nabla\hu|^2+|\nabla^2\hd|^2 \right)|\nabla\eta|^2 dx, \\
  |I_3|&\le \frac{1}{16}\int_{B_{\frac{5}{16}}}\left( |\nabla^2\hu|^2+|\Delta\nabla\hd|^2 \right)\eta^2dx
  +C\int_{B_{\frac{5}{16}}}\left( |\nabla\hu|^2+|\nabla^2\hd|^2 \right)|\nabla\eta|^2 dx.
\end{align*}
Putting these estimates into \eqref{4.47}, we obtain
\begin{equation}
\begin{split}
  &\frac{d}{dt}\int_{B_{\frac{5}{16}}}(|\nabla\hu|^2+|\nabla^2\hd|^2)\eta^2dx
  +\int_{B_{\frac{5}{16}}}\left( |\nabla^2\hu|^2+|\nabla^3\hd|^2\right)\eta^2dx\\
  &\le C\int_{\supp \eta}(|\nabla\hu|^2+|\nabla^2\hd|^2+ |\hP^{(1)}|^2+|\hu|^{3}+|\nabla\hP^{(2)}|^{\frac{3}{2}}+|\nabla^2\hP^{(2)}|^{\frac{3}{2}})dx.
  \end{split}
    \label{4.48}
\end{equation}
By Fubini's theorem, there exists $t_*\in \left[ -(\frac{5}{16})^2,-(\frac{9}{32})^2 \right]$ such that 
\begin{equation*}
  \int_{B_{\frac{5}{16}}}\left( |\nabla\hu|^2+|\nabla^2\hd|^2 \right)\eta^2(t_*)dx
  \le 100 \int_{\PP_{\frac{5}{16}}}(|\nabla\hu|^2+|\nabla^2\hd|^2)\eta^2 dxdt.
\end{equation*}
Integrating \eqref{4.48} for  $t\in [t_*, 0]$ yields that 
\begin{equation}
\begin{split}
  &\sup_{-(\frac{9}{32})^2\le t\le 0}\int_{B_{\frac{5}{16}}}(|\nabla\hu|^2+|\nabla^2\hd|^2)\eta^2(t)dx
  +\int_{-(\frac{9}{32})^2
  \le t\le 0}\int_{B_{\frac{5}{16}}}(|\nabla^2\hu|^2+|\nabla^3\hd|^2)\eta^2dxdt\\
  &\le C\int_{-\frac14}^0\int_{\supp \eta}(|\nabla\hu|^2+|\nabla^2\hd|^2+|\hP^{(1)}|^2+|\hu|^{3}+|\nabla\hP^{(2)}|^{\frac{3}{2}}+|\nabla^2\hP^{(2)}|^{\frac{3}{2}})dxdt\\
  &\qquad+C\int_{\PP_{\frac{5}{16}}}(|\nabla\hu|^2+|\nabla^2\hd|^2)\eta^2dxdt\\
  &\le C\int_{\PP_{\frac{1}{2}}}(|\nabla\hu|^2+|\nabla^2\hd|^2+|\hu|^{3}+|\hP|^{\frac{3}{2}})dxdt
  + C.
  \label{4.49}
  \end{split}
\end{equation}
For the pressure $P$, taking divergence of the equation \eqref{eqn:blowupeq}$_{1}$ yields that for any $-\frac{1}{4}\le t\le 0$, 
\begin{equation}
  -\Delta\hP_{x_i}=\dv^2 S_\alpha[\Delta\hd, \od]_{x_i} \text{ in } B_{\frac{5}{16}}. 
  \label{}
\end{equation}
We have
\begin{equation}
  \begin{split}
  \int_{\PP_{\frac{1}{4}}}|\nabla\hP|^{\frac{3}{2}}dxdt
  &\le C\int_{\PP_\frac{9}{32}}(|S_\alpha[\Delta\hd, \od]_{x_i}|^{\frac{3}{2}}+|\hP|^{\frac{3}{2}})dxdt
  \le C\int_{\PP_\frac{9}{32}}(|\nabla^3\hd|^{\frac{3}{2}}+|\hP|^{\frac{3}{2}})dxdt\\
  &\le C \int_{\PP_{\frac{9}{32}}}|\nabla^3\hd|^{2} dxdt+C\int_{\PP_{\frac{9}{32}}}|\hP|^{\frac{3}{2}}dxdt
  +C.
\end{split}
  \label{4.51}
\end{equation}
Now let $\eta$ be a cut-off function of $B_{\frac{9}{32}}$, i.e., $\eta\equiv 1$ in $B_{\frac{3}{8}}$. Then, by combining \eqref{4.49} and \eqref{4.51}, we obtain
\begin{equation}
\begin{split}
&\sup_{-(\frac{1}{4})^2\le t\le 0}\int_{B_{\frac{1}{4}}}(|\nabla\hu|^2+|\nabla\hd|^2)dx
+\int_{\PP_{\frac{1}{4}}}(|\nabla^2\hu|^2+|\nabla^3\hd|^2+|\nabla\hP|^{\frac{3}{2}})dxdt\\
&\le C\int_{\PP_{\frac{1}{2}}}(|\hu|^{3}+|\nabla\hu|^2+|\nabla^2\hd|^2+|\hP|^{\frac{3}{2}})dxdt+C.
	\label{eqn:linearfirstestimate}
	\end{split}
	\end{equation}
 It turns out that we can extend the energy method above to arbitary order. Here we sketch  the proof.  For nonnegative multiple indices $\beta, \gamma$ and $\delta$ such that $\gamma=\beta+\delta$ and $\delta$ is of order $1$, $|\beta|=k$, then $(\nabla^\beta\hu, \nabla^\gamma\hd, \nabla^\beta\hP)$ satisfies
 \begin{equation}
   \left\{
   \begin{array}{l}
     \pa_t(\nabla^\beta\hu)+\nabla(\nabla^\beta\hP)-\Delta(\nabla^\beta\hu)=-\nabla\cdot S_\alpha[\Delta(\nabla^\beta\hd), \od], \\
     \dv(\nabla^\beta\hu)=0, \\
     \pa_t(\nabla^\gamma\hd)-\Delta(\nabla^\gamma\hd)=T_\alpha[\nabla(\nabla^\gamma\hu), \od]. 
   \end{array}
   \right.
   \label{eqn:blowuphigheq}
 \end{equation}
 By differentiating $(\hP^{(1)}, \hP^{(2)})$ $(k-1)$ times we can estimate
 \begin{equation}
   \int_{\PP_{\frac{1}{2}}}|\nabla^{k-1}\hP^{(1)}|^2dxdt
   \le C\int_{\PP_{\frac{1}{2}}}|\nabla^{k+1}\hd|^2 dxdt,
   \label{}
 \end{equation}
and
\begin{equation}
\int_{\PP_{\frac{5}{16}}}|\nabla^l\hP^{(2)}|^{\frac{3}{2}}dxdt
\le C\int_{\PP_{\frac{1}{2}}}|\nabla^{k-1}\hP|^{\frac{3}{2}}dxdt
+C\int_{\PP_{\frac{1}{2}}}|\nabla^{k+1}\hd|^2 dxdt+C, \quad l=k, k+1.
\end{equation}
Multiplying \eqref{eqn:blowuphigheq}$_1$ by $(\nabla^\beta\hu)\eta^2$ and \eqref{eqn:blowuphigheq}$_3$ by $(\nabla^\gamma\hd)\eta^2$ and integrating the resulting equations over $B_\frac{1}{2}$, and by the same calculation and cancellation, we obtain
\begin{equation}
\begin{split}
  &\frac{d}{dt}\int_{B_{\frac{5}{16}}}(|\nabla^k\hu|^2+|\nabla^{k+1}\hd|^2)\eta^2 dx
  +\int_{B_{\frac{5}{16}}}(|\nabla^{k+1}\hu|^2+|\nabla^{k+2}\hd|^2)\eta^2 dx\\
  &\le C\int_{B_{\frac{5}{16}}}(|\nabla^k\hu|^2+|\nabla^{k+1}\hd|^2+|\nabla^{k-1}\hP^{(1)}|^2+|\nabla^{k-1}\hu|^{3}+|\nabla^{k}\hP^{(2)}|^{\frac{3}{2}}+|\nabla^{k+1}\hP^{(2)}|^{\frac{3}{2}})dx\\
  &\le C\int_{\PP_{\frac{1}{2}}}(|\nabla^k\hu|^2+|\nabla^{k+1}\hd|^2+|\nabla^{k-1}\hu|^{3}+|\nabla^{k-1}\hP|^{\frac{3}{2}})dxdt+C.
  \end{split}
  \label{4.56}
\end{equation}
For $P$, since 
\begin{equation}
  -\Delta(\nabla^\beta\hP)=\dv^2S_\alpha[\Delta(\nabla^\beta\hu), \od] \text{ in }B_{\frac{5}{16}},
  \label{}
\end{equation}
 we have 
 \begin{equation}
   \begin{split}
   \int_{\PP_{\frac{1}{4}}}|\nabla^k\hP|^{\frac{3}{2}}dxdt
   &
   \le C\int_{\PP_{\frac{9}{32}}}|\nabla^{k+2}\hd|^{\frac{3}{2}}dxdt
   +C\int_{\PP_{\frac{9}{32}}}|\nabla^{k-1}\hP|^{\frac{3}{2}}dxdt
   \\
   &\le C\int_{\PP_{\frac{9}{32}}}|\nabla^{k+2}\hd|^2dxdt
   +C\int_{\PP_{\frac{9}{32}}}|\nabla^{k-1}\hP|^{\frac{3}{2}}dxdt
   +C. 
 \end{split}
   \label{}
 \end{equation}
By choosing suitable $t_*$ as above, we can integrate \eqref{4.56} in $t$ to get
\begin{equation}
\begin{split}
  &\sup_{-(\frac{9}{32})^2\le t\le 0}\int_{B_{\frac{9}{32}}} (|\nabla^k\hu|^2+|\nabla^{k+1}\hd|^2)dx
  +\int_{\PP_{\frac{9}{32}}}(|\nabla^{k+1}\hu|^2+|\nabla^{k+2}\hd|^2)dxdt\\
 &\le C\int_{\PP_{\frac{1}{2}}}(|\nabla^k\hu|^2+|\nabla^{k+1}\hd|^2+|\nabla^{k-1}\hu|^3+|\nabla^{k-1}\hP|^{\frac{3}{2}})dxdt+C.
  \label{}
  \end{split}
\end{equation}
Thus, we get
\begin{equation}
\begin{split}
  &\sup_{-(\frac{1}{4})^2\le t\le 0}\int_{B_{\frac{1}{4}}}(|\nabla^k\hu|^2+|\nabla^{k+1}\hd|^2)dx
  +\int_{\PP_{\frac{1}{4}}}(|\nabla^{k+1}\hu|^2+|\nabla^{k+2}\hd|^2+|\nabla^k\hP|^{\frac{3}{2}})dxdt\\
  &\le C\int_{\PP_{\frac{1}{2}}}(|\nabla^{k-1}\hu|^3+|\nabla^{k}\hu|^2+|\nabla^{k+1}\hd|^2+|\nabla^{k-1}\hP|^{\frac{3}{2}})dxdt+C.
  \end{split}
  \label{4.60}
\end{equation}
From Sobolev's interpolation inequality, we have
\begin{equation*}
  \int_{\PP_{\frac{1}{2}}}|\nabla^{k-1}\hu|^{3}dxdt
  \le C\left\|\nabla^{k-1}\hu\right\|_{L_t^\infty L_x^2(\PP_{\frac{1}{2}})}^{6}+C \int_{\PP_{\frac{1}{2}}}(|\nabla^{k-1}\hu|^2+|\nabla^k\hu|^2)dxdt. 
\end{equation*}
Substituting this inequality in \eqref{4.60} and by suitable adjusting of the radius, we can show that
\begin{equation}
\begin{split}
  &\sup_{-(\frac{1}{4})^2\le t\le 0}\int_{B_{\frac{1}{4}}\times\{t\}}(|\nabla^k\hu|^2+|\nabla^{k+1}\hd|^2)dx+\int_{\PP_{\frac{1}{4}}}(|\nabla^{k+1}\hu|^2+|\nabla^{k+2}\hd|^2+|\nabla^k\hP|^{\frac{3}{2}})dxdt\\
  &\le  C\left(\big\|(\hu, \nabla \hd)\big\|_{L_t^\infty L_x^2\cap L_t^2 H_x^2(\PP_{\frac{1}{2}})}, \big\|\hP\big\|_{L^{\frac{3}{2}}(\PP_{\frac{1}{2}})}\right).
  \end{split}
  \label{4.61}
\end{equation}
With \eqref{4.61}, we can apply the regularity for both the linear Stokes equations  and the linear heat equation (c.f. \cite{ladyzhenska1968parabolic,temam1984ns}) to conclude that $(\hu, \hd)\in C^\infty(\PP_{\frac{1}{4}})$. Furthermore, applying the elliptic estimate for the pressure equation \eqref{eqn:hPeq}, we see that $\hP\in C^\infty(\PP_{\frac{1}{4}})$. Therefore $(\hu, \hd, \hP)\in C^\infty(\PP_{\frac{1}{4}})$ and the estimate \eqref{eqn:smoothestimate} holds. The proof is completed. 
\end{proof}

The oscillation Lemma admits the following iterations.
\begin{lemma}
  Let $(\u, \d, P), M, \varepsilon_0(M), \tau_0(M), C_0(M), z_0$ be as in  Lemma \ref{lemma:smallregularity}. Then there exist $r_0=r_0(M), \varepsilon_1=\varepsilon_1(M)>0$ such that for $0<r\le r_0$, if
	\begin{equation*}
	|\d_{z_0, r}|\le \frac{M}{2}, \qquad \Phi(z_0, r)\le \varepsilon_1^{3},
	\end{equation*}
	then  for any $k=1,2, \dots, $ we have 
	\begin{equation}
	\begin{array}[]{l}
	|\d_{z_0, \tau_0^{k-1}r}|\le M, \\
	\Phi(z_0, \tau_0^{k-1}r)\le \varepsilon_1^{3}, \\
	\Phi(z_0, \tau_0^k r)\le \frac{1}{2}\max\left\{ \Phi(z_0, \tau_0^{k-1}r), C_0 (\tau_0^{k-1}r)^{3} \right\}.
	\end{array}
	\label{}
	\end{equation}
	\label{lemma:iterative}
\end{lemma}
\begin{proof}
	We prove it by an induction on $k$. By translational invariance we may assume that $z_0=0$, and we denote $\d_r$ to be $\d_{0,r}$ for simplicity.  
	
	For $k=1$, the conclusion follows from Lemma \ref{lemma:smallregularity}, if we choose $\varepsilon_1$ such that $\varepsilon_1<\varepsilon_0$. Suppose the conclusion is true for all $k\le k_0, k_0\ge 1$, we show it remains true for $k=k_0+1$. By the inductive hypothesis
	\begin{equation*}
	\begin{array}[]{l}
	|\d_{ \tau^{k-1}_0 r}|\le M, \\
	\Phi(0, \tau_0^{k-1}r)\le \varepsilon_1^{3}, \\
	\Phi(0, \tau_0^k r)\le \frac{1}{2}\max\left\{ \Phi(0, \tau_0^{k-1}r), C_0(\tau_0^{k-1}r)^{3} \right\}\le \frac{1}{2}\max\left\{ \varepsilon_1^{3}, C_0 (\tau_0^{k-1}r)^{3} \right\}
	\end{array}
	\end{equation*}
	for all $k\le k_0$. Thus, 
	\begin{align*}
	\Phi(0, \tau_0^k r)&\le \frac{1}{2}\max\left\{ \Phi(0, \tau_0^{k-1}r), C_0(\tau_0^{k-1}r)^{3} \right\}\\
	&\le \frac{1}{2}\max\left\{ \frac{1}{2}\max\left\{ \Phi(0, \tau_0^{k-2}r), C_0(\tau_0^{k-2} r)^{3} \right\}, C_0(\tau_0^{k-1}r)^{3} \right\}\\
	&\le \cdots\le 2^{-k}\max\big\{ \Phi(0, r), \frac{C_0 r^{3}}{1-2\tau_0^{3}} \big\}\\
	&\le 2^{-k}\max\big\{ \varepsilon_1^{3}, \frac{C_0 r_0^{3}}{1-2\tau_0^{3}} \big\}, \quad \forall k\le k_0.
	\end{align*}
	Then
	\begin{align*}
	|\d_{\tau_0^{k_0}r}|&\le |\d_r|+\sum_{k=1}^{k_0}\Big|\d_{\tau_0^k}-\d_{\tau_0^{k-1}r}\Big|\\
	&\le \frac{M}{2}+\sum_{k=1}^{k_0}\Big( \fint_{\PP_{\tau_0^k r}(0)}|\d-\d_{\tau^{k-1}_0 r}|^{6} \Big)^{\frac{1}{6}}\\
	&\le\frac{M}{2}+\sum_{k=1}^{k_0} \Phi(0, \tau_0^{k-1}r)^{\frac{1}{3}} \\
	&\le \frac{M}{2}+\sum_{k=1}^{k_0} 2^{-\frac{1}{3}(k-1)} \max\big\{ \varepsilon_1, \big( \frac{C_0 r_0^{3}}{1-2\tau_0^{3}} \big)^{\frac{1}{3}} \big\}\\
	&\le 
	\frac{M}{2}+\frac{1}{1-2^{-\frac{1}{3}}}\max\big\{ \varepsilon_1, \big( \frac{C_0 r_0^{3}}{1-2\tau_0^{3}} \big)^{\frac{1}{3}} \big\}. 
	\end{align*}
	If we choose sufficiently small $r_0=r_0(M)$, $\varepsilon_1=\varepsilon_1(M)$, we see 
	\begin{equation*}
	\begin{array}[]{l}
	|\d_{\tau_0^{k_0} r}|\le M, \\
	\Phi(0, \tau_0^{k_0} r)\le \varepsilon_1^{3}\le \varepsilon_0^{3}.
	\end{array}
	\end{equation*}
	It follows directly from Lemma \ref{lemma:smallregularity} with $r$ replaced by $\tau_0^kr$ that
	\begin{equation*}
	\Phi(0, \tau_0^{k+1}r)\le \frac{1}{2}\max\left\{ \Phi(0, \tau_0^k r), C_0 (\tau_0^kr)^{3} \right\}.
	\end{equation*}
This completes the proof. \end{proof}
The local boundedness of the solutions can be obtained by utilizing the Riesz potential estimates between
Morrey spaces as in the following lemma.
\begin{lemma}
  For any $M>0$, there exists $\varepsilon_2>0$, depending on $M$, such that if $(\u, \d, P)$ is a suitable weak solution of \eqref{eqn:ELmodel1} in 
  ${\R^3}\times(0,\infty)$, which satisfies, for $z_0=(x_0, t_0)\in {\R^3}\times(r_0^2, \infty)$ 
  \begin{equation}
    |\d_{z_0, r_0 }|\le \frac{M}{4}, \text{ and }\Phi(z_0, r_0)\le \varepsilon_2^{3},
    \label{}
  \end{equation}
  then for any $1<p<\infty$, $(\u, \nabla\d)\in L^p(\PP_{\frac{r_0}{4}}(z_0))$, 
  {{$\d\in C^\theta(\PP_{\frac{r_0}{2}}(z_0))$}} and
 \begin{equation}
 {{|\d|\le M \text{ in }\PP_{\frac{r_0}{2}}(z_0), \quad	[\d]_{C^\theta(\PP_{\frac{r_0}{2}}(z_0))}\le C(\theta,  M)(\varepsilon_1+r_0). }}
 \end{equation}
    \begin{equation}
    \left\|(\u, \nabla\d)\right\|_{L^p(\PP_{\frac{r_0}{4}}(z_0))}\le C(p, M)(\varepsilon_1+r_0),
    \label{eqn:dHolder}
  \end{equation}
  where $\varepsilon_1$ is the constant in Lemma \ref{lemma:iterative}.
  \label{lemma:localboundedness}
\end{lemma}
\begin{proof}
  Let $\varepsilon_2=\min\big\{ \big( \frac{M}{4} \big), 2^{-\frac{11}{6}}\varepsilon_1(M) \big\}$. 
  For any $z\in \PP_{\frac{r_0}{2}}(z_0)$,
  \begin{align*}
    |\d_{z, \frac{r_0}{2}}|&\le \big|\d_{z, \frac{r_0}{2}}-\d_{z_0, r_0}\big|+|\d_{z_0, r_0}|\\
    &\le \fint_{\PP_{\frac{r_0}{2}}(z)}|\d-\d_{z_0, r_0}|+\frac{M}{4}
    \le \varepsilon_2+\frac{M}{4}\le \frac{M}{2}. 
  \end{align*}
  Meanwhile, 
  \begin{equation*}
  \begin{split}
    &\big( \fint_{\PP_{\frac{r_0}{2}}(z)}|\d-\d_{z, \frac{r_0}{2}}|^{6}dxdt \big)^{\frac{1}{2}}\\
    &\le \big(2^{5}\fint_{\PP_{\frac{r_0}{2}}(z)}|\d-\d_{z_0, r_0}|^{6}dxdt+2^{5}|\d_{z_0, r_0}-\d_{z, \frac{r_0}{2}}|^{6} \big)^{\frac{1}{2}}\\
    &\le\big( 2^{10}\fint_{\PP_{r_0}(z_0)}|\d-\d_{z_0, r_0}|^{6}dxdt+2^{5}\fint_{\PP_{\frac{r_0}{2}(z)}}|\d-\d_{z_0, r_0}|^{6}dxdt \big)^{\frac{1}{2}}\\
    &\le 2^{\frac{11}{2}} \big( \fint_{\PP_{r_0}(z_0)}|\d-\d_{z_0, r_0}|^{6}dxdt \big)^{\frac{1}{2}},
    \end{split}
    \label{}
  \end{equation*}
  Hence we get that
  \begin{equation*}
    \Phi(z, \frac{r_0}{2})\le 2^{\frac{11}{2}}\Phi(z_0, r_0)\le 2^{\frac{11}{2}} \varepsilon_2^3\le \varepsilon_1^{3}.
  \end{equation*}
  Then we deduce from Lemma \ref{lemma:iterative} that for any $k=1,2, \dots$, 
  \begin{equation}
    \begin{array}[]{l}
      |\d_{z, \tau_0^{k-1}\frac{r_0}{2}}|\le M,\\ 
      \Phi(z, \tau_0^k \frac{r_0}{2})\le \frac{1}{2}\max\left\{ \Phi(z, \tau_0^{k-1}\frac{r_0}{2}), C_0(\tau_0^{k-1}r)^{3} \right\}.
    \end{array}
    \label{}
  \end{equation}
  By Lebesgue's differentiation theorem, we have $|\d|\le M$ a.e. in $\PP_{\frac{r_0}{2}}(z_0)$. Furthermore, we have
  \begin{equation*}
    \Phi(z, \frac{\tau_0^k r_0}{2})\le 2^{-k}\max\Big\{ \Phi(z, \frac{r_0}{2}), \frac{C_0 r_0^{3}}{1-2\tau_0^{3}} \Big\}.
  \end{equation*}
  Therefore for $\theta_0=\frac{\ln 2}{3|\ln \tau_0|}\in (0, \frac{1}{3})$, it holds for any $0<s<\frac{r_0}{2}$ and $z\in\PP_{\frac{r_0}{2}}(z_0)$, 
{{  \begin{equation}
    \Phi(z, s)\le C(r_0^3+\varepsilon_1^{3})\big( \frac{s}{r_0} \big)^{3\theta_0}. 
    \label{eqn:decay1}
  \end{equation}}}
  By the Campanato theory, $\d\in C^\theta(\PP_{\frac{r_0}{4}}(z_0))$ and \eqref{eqn:dHolder} holds.
 Now for $\p \in C_0^\infty(\PP_{\frac{r_0}{2}})(z_0)$, from \eqref{eqn:localueq}, \eqref{eqn:localdeq} we can derive the following local energy inequality:
  \begin{equation}
    \begin{split}
      &\frac{1}{2}\int_{\R^3}(|\u|^2+|\nabla\d|^2)\p(x,t)dx+\int_{0}^{t}\int_{\R^3}(|\nabla\u|^2+|\Delta\d|^2)\p(x,s)dxds\\
      &\le\int_{0}^{t}\int_{\R^3}\frac{1}{2}(|\u|^2+|\nabla\d|^2)(\pa_t\p+\Delta\p)(x,s)dxds\\
      &+\int_{0}^{t}\int_{\R^3}[\frac{1}{2}(|\u|^2+2P)\u\cdot \nabla\p+\nabla\d\odot\nabla\d:\u\otimes\nabla\p](x,s)dxds\\
      &+\int_{0}^{t}\int_\R^3 (\nabla\d\odot\nabla\d-|\nabla\d|^2I_3):\nabla^2\p(x,s)dxds\\
      &+\int_{0}^{t}\int_{\R^3}[S_\alpha[\Delta\d, \d]:\u\otimes\nabla\p+T_\alpha[\nabla\u, \d]\cdot (\nabla\p\cdot \nabla\d)](x,s)dxds\\
      &+\int_{0}^{t}\int_{\R^3}\nabla\cdot S_\alpha[\f(\d), \d]\cdot \u\p(x,s)dxds\\
      &+\int_{0}^{t}\int_{\R^3}(\u\cdot \nabla\d)\cdot \f(\d)\p(x,s)dxds-\int_{0}^{t}\int_{\R^3}\nabla\f(\d):\nabla\d\p(x,s)dxds.
    \end{split}
    \label{eqn:LocalenergywithoutF}
  \end{equation}
 Let $\p\in C_0^\infty(\PP_{2s}(z))$ be a cut-off function of $\PP_s(z)$.  Replacing $\p$ by $\p^2$ in \eqref{eqn:LocalenergywithoutF}, we can show that for $0<s<\frac{r_0}{2}$, 
  \begin{equation}
    \begin{split}
      &s^{-1}\int_{\PP_s(z)}(|\nabla\u|^2+|\Delta\d|^2)dxdt\\
      &\le C[(2s)^{-3}\int_{\PP_{2s}(z)}(|\u|^2+|\nabla\d|^2)dxdt
      +(2s)^{-2}\int_{\PP_{2s}(z)}(|\u|^3+|\nabla\d|^3+|P|^{\frac{3}{2}})dxdt]\\
      &\le C(r_0^3+\varepsilon_1^3)\big( \frac{s}{r_0} \big)^{2\theta_0}. 
    \end{split}
    \label{eqn:decay2}
  \end{equation}

    Now we are ready to perform the Riesz potential estimate.
    For any open set $U\subset{\R^3}\times\R, 1\le p<\infty$, define the Morrey space $M^{p,\lambda}(U)$ by
    \begin{equation*}
      M^{p,\lambda}(U):=\left\{ f\in L_{\loc}^p(U): \left\|f\right\|_{M^{p,\lambda}(U)}^p=\sup_{z\in U, r>0}r^{\lambda-5}\int_{\PP_r(z)}|f|^p dxdt<\infty \right\}. 
    \end{equation*}
    It follows from \eqref{eqn:decay1} and \eqref{eqn:decay2} that there exists $\alpha\in(0,1)$ such that 
    \begin{equation*}
      (\u, \nabla\d)\in M^{3, 3(1-\alpha)}\big( \PP_{\frac{r_0}{2}}(z_0) \big), 
      P\in M^{\frac{3}{2}, 3(1-\alpha) }\big( \PP_{\frac{r_0}{2}}(z_0) \big),
      (\nabla\u, \nabla^2\hd)\in M^{2, 4-2\alpha}\big( \PP_{\frac{r_0}{2}}(z_0) \big). 
    \end{equation*}
    Write $\d$ equation in \eqref{eqn:ELmodel1} as
    \begin{equation}
      \pa_t\d-\Delta\d=-\u\cdot \nabla\d+T_\alpha[\nabla\u,\d]-\f(\d) \in M^{\frac{3}{2}, 3(1-\alpha)}\big( \PP_{\frac{r_0}{2}}(z_0) \big).
      \label{}
    \end{equation}
    Let $\eta\in C_0^\infty(\R^4)$ be such that $0\le \eta\le 1$, $\eta=1$ in $\PP_{\frac{r_0}{2}}(z_0)$, $|\pa_t\eta|+|\nabla^2\eta|\le Cr_0^2$.  Set $\w=\eta^2(\d-\d_{z_0, \frac{r_0}{2}})$. Then 
    \begin{equation}
      \pa_t \w-\Delta \w=F, \quad F:=\eta^2 (\pa_t\d-\Delta\d)+(\pa_t\eta^2-\Delta\eta^2)(\d-\d_{z_0, \frac{r_0}{2}})-2\nabla\eta^2\cdot \nabla\d. 
      \label{}
    \end{equation}
    We can check that $F\in M^{\frac{3}{2}, 3(1-\alpha)}(\R^4)$ and satisfies
    \begin{equation}
      \left\|F\right\|_{M^{\frac{3}{2}, 3(1-\alpha)}(\R^4)}\le C(r_0+\varepsilon_1). 
      \label{}
    \end{equation}
    Let $\Gamma$ denote the heat kernel in ${\R^3}$.  Then 
\begin{equation*}
  |\nabla\Gamma(x, t)|\le C\delta^{-4}( (x, t), (0, 0)), \forall (x,t)\neq (0, 0),
\end{equation*}
where $\delta(\cdot , \cdot )$ denotes the parabolic distance on $\R^4$. By the Duhamel formula, we have that 
\begin{equation}
  |\w(x, t)|\le \int_{0}^{t}\int_{\R^3}|\nabla\Gamma(x-y, t-s)||F(y,s)|dyds\le C\mathcal{I}_1(|F|)(x,t), 
  \label{}
\end{equation}
where $\mathcal{I}_\beta$ is the parabolic Riesz potential of order $\beta$ on $\R^4$, $0\le\beta\le 5$, defined by 
\begin{equation*}
  \mathcal{I}_\beta(g)(x,t)=\int_{\R^4}\frac{|g(y,s)|}{\delta^{5-\beta}( (x,t), (y,s))}dyds, \forall g\in L^2(\R^4).
\end{equation*}
Applying the Riesz potential estimates \cite{HuangWang2010Note}, we conclude that $\nabla \w\in M^{\frac{3(1-\alpha)}{1-2\alpha},3(1-\alpha)}(\R^4)$ and 
\begin{equation}
  \left\|\nabla \w\right\|_{M^{\frac{3(1-\alpha)}{1-2\alpha}, 3(1-\alpha)}(\R^4)} \le C\left\|F\right\|_{M^{\frac{3}{2}, 3(1-\alpha)}(\R^4)}\le C(r_0+\varepsilon_1).
  \label{}
\end{equation}
Since $\lim_{\alpha\uparrow \frac{1}{2}}\frac{3(1-\alpha)}{1-2\alpha}=\infty$, we conclude that for any $1<p<\infty$, $\nabla \w\in L^p(\PP_{r_0}(z_0))$ and 
\begin{equation*}
  \left\|\nabla \w\right\|_{L^p(\PP_{\frac{r_0}{2}}(z_0))}\le C(p) (r_0+ \varepsilon_1).
\end{equation*}
Since $\d-\w$ solves 
\begin{equation*}
  \pa_t(\d-\w)-\Delta(\d-\w)=0 \text{ in }\PP_{\frac{r_0}{4}(z_0)},
\end{equation*}
it follows from the theory of heat equations that  $\nabla(\d-\w)\in L^\infty(\PP_{\frac{r_0}{4}}(z_0))$. Therefore 
for any $1<p<\infty$, $\d\in L^p(\PP_{\frac{r_0}{4}}(z_0)$, and  
\begin{equation*}
  \left\|\nabla\d\right\|_{L^p(\PP_{\frac{r_0}{4}}(z_0))}\le C(p)(r_0+\varepsilon_1). 
\end{equation*}

We now proceed with the estimation of $\u$. Let $\vv:{\R^3}\times(0,\infty)\mapsto{\R^3}$ solve the Stokes equation:
\begin{equation}
  \left\{
  \begin{array}{l}
    \pa_t\vv-\Delta\vv+\nabla P=-\dv[\eta^2(\u\otimes\u+\nabla\d\odot\nabla\d-\frac{1}{2}|\nabla\d|^2I_3)]\\
    \qquad\qquad\qquad\qquad+\dv\{\eta^2(F(\d)-F(\d)_{z_0, \frac{r_0}{2}})I_3\}\\
    \qquad\qquad\qquad\qquad-\dv\{\eta^2( S_\alpha[\Delta\d-\f(\d),\d]+S_\alpha[\f(
    \d), \d]_{z_0, \frac{r_0}{2}})\},\\
    \nabla\cdot \vv=0, \\
    \vv(\cdot , 0)=0.
  \end{array}
  \right.
  \label{}
\end{equation}
By using the Oseen kernel, an estimate of $\vv$ can be given by 
\begin{equation}
  |\vv(x,t)|\le C \mathcal{I}_1(|X|)(x, t), \forall (x,t)\in{\R^3}\times(0,\infty),
  \label{}
\end{equation}
where
\begin{equation*}
\begin{split}
  X=\eta^2\big[ \u\otimes\u+\big( \nabla\d\odot\nabla\d-\frac{1}{2}|\nabla\d|^2I_3 \big)-(F(\d)-F(\d)_{z_0, \frac{r_0}{2}})I_3\\
 +S_\alpha[\Delta\d-\f(\d),\d]+S_\alpha[\f(\d),\d]_{z_0, \frac{r_0}{2}} \big].
 \end{split}
\end{equation*}
As above, we can check that $X\in M^{\frac{3}{2}, 3(1-\alpha)}(\R^4)$ and 
\begin{equation*}
  \left\|X\right\|_{M^{\frac{3}{2}, 3(1-\alpha)}(\R^4)}\le C(r_0+\varepsilon_1).
\end{equation*}
Hence we conclude that $\vv\in M^{\frac{3(1-\alpha)}{1-2\alpha}, 3(1-\alpha)}(\R^4)$, and
\begin{equation}
  \left\|\vv\right\|_{M^{\frac{3(1-\alpha)}{1-2\alpha}, 3(1-\alpha)}(\R^4)}\le C\left\|X\right\|_{M^{\frac{3}{2}, 3(1-\alpha)}(\R^4)}\le C(r_0+\varepsilon_1). 
  \label{}
\end{equation}
As $\alpha\uparrow\frac{1}{2}$, $\frac{3(1-\alpha)}{1-2\alpha}\to \infty$,  we conclude that for any $1<p<\infty$, $\vv\in L^p(\PP_{\frac{r_0}{2}}(z_0))$.
Since
\begin{equation*}
  \pa_t(\u-\vv)-\Delta(\u-\vv)+\nabla P=0, \dv(\u-\vv)=0 \text{ in }\PP_{\frac{r_0}{2}(z_0)},
\end{equation*}
we have that $\u-\vv\in L^\infty(\PP_{\frac{r_0}{4}}(z_0))$. Therefore for any $1<p<\infty$, $\u\in L^p(\PP_{\frac{r_0}{4}}(z_0))$ and 
\begin{equation*}
  \left\|\u\right\|_{L^p(\PP_{\frac{r_0}{4}}(z_0))}\le C(p) (r_0+\varepsilon_1).
\end{equation*} 
\end{proof}

For the rest of this section, we will establish the higher order regularity of $\eqref{eqn:ELmodel1}$. Again we prove it via a high order energy method which has been employed by Huang--Lin--Wang \cite{HuangLinWang2014EricksenLeslieTwoDimension} for general Ericksen--Leslie systems in dimension two, and Du--Hu--Wang \cite{du2019suitable} for co-rotational Beris--Edwards model in dimension three. 
\begin{lemma}
  Under the same assumption as Lemma \ref{lemma:localboundedness}, we have that for any $k\ge 0$, $(\nabla^k\u, \nabla^{k+1}\d)\in (L_t^\infty L_x^2\cap L_t^2 H_x^1)\big( \PP_{\frac{1+2^{-(k+1)}}{2}r_0}(z_0) \big)$ and the following estimates hold
  \begin{equation}
    \begin{split}
      &\sup_{t_0-\big( \frac{1+2^{-(k+1)}}{2}r_0 \big)^2\le t\le t_0}\int_{B_{\frac{1+2^{-(k+1)}}{2}r_0}(x_0)}(|\nabla^k\u|^2+|\nabla^{k+1}\d|^2)dx\\
      &\quad+\int_{\PP_{\frac{1+2^{-(k+1)}}{2}r_0}(z_0)}\left( |\nabla^{k+1}\u|^2+|\nabla^{k+2}\d|^2 +|\nabla^k P|^{\frac{3}{2}}\right)dxdt\\
      &\le C(k, r_0)\varepsilon_1.
    \end{split}
    \label{eqn:HigherOrderEnergy}
  \end{equation}
  In particular, $(\hu, \hd)$ is smooth in $\PP_{\frac{r_0}{4}}(z_0)$. 
  \label{lemma:HigherOrderRegularity}
\end{lemma}
\begin{proof}
  For simplicity, assume $z_0=(0, 0)$ and $r_0=2$. \eqref{eqn:HigherOrderEnergy} can be proved by an induction on $k$. It is clear that when $k=0$, \eqref{eqn:HigherOrderEnergy} follows directly from the local energy inequality \eqref{eqn:LocalenergywithoutF}. Here we indicate to how to proof \eqref{eqn:HigherOrderEnergy} for $k\ge 1$. Suppose that \eqref{eqn:HigherOrderEnergy} holds for $k\le l-1$, we want to show that \eqref{eqn:HigherOrderEnergy} also holds for $k=l$. From the
  induction hypothesis, we have that for $0\le k\le l-1$, 
  \begin{equation}
  \begin{split}
    &\sup_{-\left( 1+2^{-(k+1)} \right)^2\le t\le 0}\int_{B_{1+2^{-(k+1)}}}(|\nabla^k\u|^2+|\nabla^{k+1}\d|^2)dx\\
    &+\int_{\PP_{1+2^{-(k+1)}}}\left( |\nabla^{k+1}\u|+|\nabla^{k+2}\d|^2+|\nabla^k P|^{\frac{3}{2}} \right)dxdt\le C(l) \varepsilon_1. 
    \end{split}
    \label{4.79}
  \end{equation}
  Hence by the Sobolev embedding we have
  \begin{equation}
    \int_{\PP_{1+2^{-l}}}(|\nabla^{l-1}\u|^{\frac{10}{3}}+|\nabla^{l}\d|^{\frac{10}{3}})dxdt
    \le C(l)\varepsilon_1,
        \label{}
  \end{equation}
and for $0\le k\le l-2$, by the Sobolev-interpolation inequality as in \eqref{eqn:L10interpolation} we have
\begin{equation}
  \int_{\PP_{1+2^{-(k+1)}}}(|\nabla^{k}\u|^{10}+|\nabla^{k+1}\d|^{10})dxdt\le C(l)\varepsilon_1.
  \label{}
\end{equation}
Also, for $1\le j\le l-1$, we have
\begin{equation}
\begin{split}
  &\int_{-(1+2^{-j})^2}^0 \left\|(\nabla^j\u, \nabla^{j+1}\d)\right\|_{L^3(B_{1+2^{-j}})}^4dt\\
  &\le \int_{-(1+2^{-j})^2}^{0}\left\|(\nabla^j\u, \nabla^{j+1}\d)\right\|^2_{L^2(B_{1+2^{-j}})}\left\|(\nabla^{j}\u, \nabla^{j+1}\d)\right\|^2_{L^6(B_{1+2^{-j}})} dt\\
  &\le  \left\|(\nabla^j\u, \nabla^{j+1}\d)\right\|_{L_t^\infty L_x^2(\PP_{1+2^{-j}})}^2 \left\|(\nabla^j\u, \nabla^{j+1}\d)\right\|_{L_t^2 H_x^1(\PP_{1+2^{-j}})}^2 \le C(l)\varepsilon_1 
  \label{}
  \end{split}
\end{equation}
By Lemma \ref{lemma:localboundedness} we also have that any $i\in \mathbb{N}^+$ and $1<p<\infty$, 
  \begin{equation}
  \begin{split}
 & \|\d\|_{L^\infty(\PP_2)}\le M,\quad  [\d\big]_{C^\theta(\PP_2)}+[D_{\d}^i \f(\d)]_{C^\theta(\PP_2)}\le C(i, M) \varepsilon_0,\\
&\big\|(\u, \nabla\d)\big\|_{L^p(\PP_2)}\le C(p) \varepsilon_1.
\end{split}
    \label{eqn:Lpbound}
  \end{equation}
  Notice that $\nabla^{l-1} P$ satisfies
  \begin{equation}
 \begin{split}
    -\Delta \nabla^{l-1} P&=\dv^2\Big[\nabla^{l-1}\Big(\u\otimes\u+\nabla\d\odot\nabla\d-\frac{1}{2}|\nabla \d|^2I_3\\
    &-(F(\d)I_3-\fint_{\PP_2}F(\d)I_3)
    +S_\alpha[\Delta\d-\f(\d), \d]+\fint_{\PP_2}S_\alpha[\f(
    	\d), \d]\Big)\Big], 
    \label{}
    \end{split}
  \end{equation}
  Now let $\zeta\in C_0^\infty(B_{1+2^{-l}})$ be a cut-off function of $B_{1+2^{-(l+1)}+3^{-(l+1)}}$, and $P^{(1)}(\cdot , t):\R^3\to \R$, $-(1+2^{-1})^2\le t\le 0$, 
  \begin{equation}
  \begin{split}
    P^{(1)}(x, t)&:=\int_{\R^3}\nabla_x^2 G(x-y)\zeta(y)\Big[\u\otimes\u+\nabla\d\odot\nabla\d-\frac{1}{2}|\nabla\d|^2I_3\\
    &-(F(\d)I_3-\fint_{\PP_{2}}F(\d)I_3)+S_\alpha[\Delta\d-\f(\d), \d]+\fint_{\PP_{2}}S_\alpha[\f(\d),\d]\Big](y)dy, 
    \label{}
    \end{split}
  \end{equation}
  and $P^{(2)}(\cdot , t):=(P-P^{(1)})(\cdot , t)$. For $P^{(1)}$, we have that 
  \begin{equation*}
  \begin{split}
    \nabla^{l-1}P^{(1)}(x)&=\int_{\R^3}\nabla_x^2 G(x-y)\nabla^{l-1}\Big[\eta\Big( \u\otimes\u+\nabla\d\odot\nabla\d-\frac{1}{2}|\nabla\d|^2 I_3\\
    &-(F(\d)I_3-\fint_{\PP_{2}} F(\d)I_3)+S_\alpha[\Delta\d-\f(\d), \d] +\fint_{\PP_{2}}S_\alpha[\f(\d), \d])\Big) \Big](y) dy.
    \end{split}
  \end{equation*}
  By Calderon-Zygmund's singular integral estimate, with bounds \eqref{4.79}-\eqref{eqn:Lpbound} we can show that
  \begin{equation}
    \int_{\PP_{1+2^{-l}}}|\nabla^{l-1}P^{(1)}|^2dxdt\le C(l)\varepsilon_1.
    \label{}
  \end{equation}
  We see that $P^{(2)}$ satisfies
  \begin{equation}
    -\Delta P^{(2)}=0 \text{ in }B_{1+2^{-(l+1)}+3^{-(l+1)}}. 
    \label{}
  \end{equation}
  Then we derive from the regularity of harmonic function that for $1\le j\le 2l$, 
  \begin{align*}
    \int_{\PP_{1+2^{-(l+1)}+5^{-(l+1)}}}|\nabla^{j}P^{(2)}|^{\frac{3}{2}}dxdt
    &\le C\int_{\PP_{1+2^{-(l+1)}+4^{-(l+1)}}}|\nabla^{l-1} P^{(2)}|^{\frac{3}{2}}dxdt\\
    &\le C\int_{\PP_{1+2^{-l}}}|\nabla^{l-1}P|^{\frac{3}{2}}dxdt
    +C\int_{\PP_{1+2^{-l}}}|P^{(1)}|^{\frac{3}{2}}dxdt\\
    &\le C(l)\varepsilon_1. 
  \end{align*}
  Now take $l-$th order spatial derivative of the equation \eqref{eqn:ELmodel1}$_1$, we have\footnote{Strictly speaking, we need to take finite difference quotient $D_h^i\nabla^{l-1}$ of \eqref{eqn:ELmodel1}$_1$ and then sending $h\to 0$.}
  \begin{equation}
  \begin{split}
    &\pa_t(\nabla^l\u)+\nabla^l\nabla\cdot (\u\otimes\u)+\nabla^l\nabla P-\nabla^l\Delta\u\\&\quad=-\nabla^l\nabla\cdot \left[ \nabla\d\odot\nabla\d-\frac{1}{2}|\nabla\d|I_3-F(\d)I_3+S_\alpha[\Delta\d-\f(\d), \d] \right].
    \label{eqn:lthueq}
    \end{split}
  \end{equation}
  Let $\eta\in C_0^\infty(B_{1+2^{-l}})$. Multiplying \eqref{eqn:lthueq} by $\nabla^l \u\eta^2$ and integrating over $B_2$, we obtain\footnote{Strictly speaking, we need to multiply the equation by $D_h^i\nabla^{l-1}\u\eta^2$.}
  \begin{equation}
  \begin{split}
    &\frac{d}{dt}\int_{B_2}\frac{1}{2}|\nabla^l \u|^2\eta^2dx+\int_{B_2}|\nabla^{l+1}\u|^2\eta^2dx\\
    &=\int_{B_2}[\nabla^l(\u\otimes\u):\nabla\nabla^l\u\eta^2+\nabla^l(\u\otimes\u):\nabla^l\u\otimes\nabla (\eta^2)]dx\\
    &\quad+\int_{B_2}\nabla^l P\cdot \nabla^l\u\cdot \nabla(\eta^2)dx
    -\int_{B_2}\nabla\nabla^{l}\u:\nabla^{l}\u\otimes\nabla(\eta^2)dx\\
    &\quad+\int_{B_2}\nabla^l\left[ \nabla\d\odot\nabla\d-\frac{1}{2}|\nabla\d|^2-F(\d)I_3-S_\alpha[\f(\d), \d] \right]:\nabla(\nabla^l\u\eta^2)dx\\
    &\quad+\int_{B_2}\nabla^l S_\alpha[\Delta\d, \d]:(\nabla \nabla^l\u\eta^2+\nabla^l\u\otimes \nabla (\eta^2))dx
    \\&:=I_1+I_2+I_3+I_4+I_5.
    \end{split}
    \label{}
  \end{equation}
Now we have the following estimate:
\begin{align*}
  |I_1|&\lesssim \int_{B_2}\left[ |\u||\nabla^{l}\u|+\sum_{j=1}^{l-1}|\nabla^j\u||\nabla^{l-j}\u| \right](|\nabla^{l+1}\u|\eta^2+|\nabla^l\u|\eta|\nabla\eta|)dx\\
  &\le \frac{1}{32} \int_{B_2}|\nabla^{l+1}\u|^2\eta^2dx
  +C\int_{B_2}|\u|^2|\nabla^l\u|^2\eta^2dx\\
  & +C\int_{B_2}\sum_{j=1}^{l-1}|\nabla^j\u|^2|\nabla^{l-j}\u|^2\eta^2dx+C\int_{\supp \eta}|\nabla^l\u|^2dx,\\
  |I_2|&\lesssim \int_{B_2}[|\nabla^{l-1}P^{(1)}|(|\nabla^{l+1}\u|\eta|\nabla\eta|+|\nabla^l\u||\nabla^2(\eta^2)|)+|\u||\nabla^l(\nabla^l P^{(2)} \nabla\eta^2)|)]dx\\
  &\le \frac{1}{32}\int_{B_2}|\nabla^{l+1}\u|^2\eta^2dx
  +C\int_{\supp\eta}(|\nabla^{l-1}P^{(1)}|^2+|\nabla^l\u|^2)dx\\
  &+C\int_{\supp \eta}(|\u|^{3}+|P^{(2)}|^{\frac{3}{2}})dx\\
\\
  |I_3|&\lesssim \int_{B_2}|\nabla^{l+1}\u|\eta |\nabla^l\u||\nabla\eta|dx
  \le \frac{1}{32}\int_{B_2}|\nabla^{l+1}\u|^2\eta^2dx
  + C\int_{\supp \eta}|\nabla^l \u|^2dx,
  \\
  |I_4|&\le \int_{B_2}\Big( |\nabla^{l+1}\d||\nabla\d|+\sum_{j=1}^{l-1}|\nabla^{j+1}\d||\nabla^{l+1-j}\d|+|\nabla^lF(\d)|+|\nabla^{l}(S_\alpha[f(\d),\d])| \Big)\\
  &\quad \times(|\nabla^{l+1}\u|\eta^2+|\nabla^l\u||\nabla(\eta^2)|)dx\\
  &\le \frac{1}{32}\int_{B_2}|\nabla^{l+1}\u|^2\eta^2dx
  +C\int_{B_2}(|\nabla^{l+1}\d|^2|\nabla\d|^2\eta^2+\sum_{j=1}^{l-1}|\nabla^{j+1}\d|^2|\nabla^{l+1-j}\d|^2\eta^2)dx\\
  &\qquad+C\int_{B_2} (|\nabla^l F(\d)|^2\eta^2+|\nabla^l S_{\alpha}[\f(\d),\d]|^2\eta^2)dx
   +C\int_{\supp \eta}|\nabla^l\u|^2dx.
\end{align*}
For $I_5$, set $A^l_\alpha:=S_\alpha[\nabla^l\Delta\d, \d]$, and $B^l_\alpha:=\nabla^l S_\alpha[\Delta\d,\d]-A^l_\alpha$, then we have
\begin{align*}
  I_5&=\int_{B_2}[A_\alpha^l: \nabla\nabla^l\u\eta^2+ B^l_\alpha:\nabla\nabla^l\u \eta^2+A_\alpha^l:\nabla^l\u\otimes\nabla(\eta^2)+B_\alpha^l:\nabla^l\u\otimes\nabla(\eta^2)]dx
  \\&=:I_{51}+I_{52}+I_{53}+I_{54}.
\end{align*}
Then we get
\begin{align*}
  |I_{52}|&\le \frac{1}{32}\int_{B_2}|\nabla^{l+1}\u|^2\eta^2dx
  +C\int_{B_2}|\nabla\d|^2|\nabla^{l+1}\d|^2\eta^2dx\\
  &\ \ \ +C\int_{B_2}\sum_{j=1}^{l-1}|\nabla^{j+1}\d|^2|\nabla^{l+1-j}\d|^2\eta^2dx, \\
  |I_{53}|&\le \frac{1}{32}\int_{B_2}|\nabla^{l+2}\d|^2\eta^2dx
  +C\int_{\supp \eta}|\nabla^l\u|^2dx, \\
  |I_{54}|&\lesssim \int_{\supp \eta}|\nabla^l\u|^2 dx
  +\int_{B_2}|\nabla\d|^2|\nabla^{l+1}\d|^2\eta^2 dx
  +\int_{B_2}\sum_{j=1}^{l-1}|\nabla^{j+1}\d|^2|\nabla^{l+1-j}\d|^2\eta^2 dx.
\end{align*}
  Now we take $(l+1)$-th order spartial derivative of the equation \eqref{eqn:ELmodel1}$_3$, we have
  \begin{equation}
    \pa_t(\nabla\nabla^{l}\d)+\nabla\nabla^{l}(\u\cdot \nabla\d)-\nabla\nabla^{l}T_\alpha[\nabla\u, \d]=\Delta\nabla\nabla^{l}\d-\nabla\nabla^{l}\f(\d).
    \label{eqn:lthdeq}
  \end{equation}
  Multiplying \eqref{eqn:lthdeq} by $\nabla\nabla^{l}\d\eta^2$ and integrating over $B_2$, we obtain\footnote{Strictly speaking, we need to multiply the equation by $D_h^i \nabla^l \d \eta^2$.}
  \begin{equation}
  \begin{split}
    &\frac{d}{dt}\int_{B_2}\frac{1}{2}|\nabla^{l+1}\d|^2\eta^2dx
    +\int_{B_2}|\nabla^{l+2}\d|^2\eta^2dx\\
    &=\int_{B_2}\nabla^{l}(\u\cdot \nabla\d)\cdot\nabla \cdot (\nabla\nabla^{l}\d\eta^2)dx\\
    &\quad-\int_{B_2}[\nabla^l T_\alpha[\nabla\u, \d]\cdot \Delta\nabla^l \d\eta^2+ \nabla^l T_\alpha[\nabla\u, \d]\cdot (\nabla(\eta^2)\cdot \nabla \nabla^l \d)]dx\\
    &\quad-\int_{B_2}\nabla\nabla^l\f(\d):\nabla \nabla^l\d\eta^2dx=:K_1+K_2+K_3.
    \end{split}
    \label{}
  \end{equation}
Then we have the following estimates:
\begin{align*}
  |K_1|&\lesssim\int_{B_2}\big[ |\nabla\d||\nabla^l\u|+|\u||\nabla^{l+1}\d|+\sum_{j=1}^{l-1}|\nabla^j\u||\nabla^{l-j+1}\d| \big]\big( |\nabla^{l+2}\d|\eta^2+|\nabla^{l+1}\d|\eta|\nabla\eta| \big)dx\\
  &\le \frac{1}{32}\int_{B_2}|\nabla^{l+2}\d|^2\eta^2dx
  +C\int_{B_2}|\nabla\d|^2|\nabla^l\u|^2\eta^2dx
  +C\int_{B_2}|\u|^2|\nabla^{l+1}\d|^2\eta^2dx\\
  &\qquad+C\int_{B_2}\sum_{j=1}^{l-1}|\nabla^j\u|^2|\nabla^{l-j+1}\d|^2\eta^2dx
  +C\int_{\supp \eta}|\nabla^{l+1}\d|^2dx,\\
  |K_3|&\lesssim \int_{B_2} |\nabla^{l+1}\d|^2\eta^2dx+\int_{B_2}|\nabla^{l+1} \f(\d)|^2\eta^2dx. 
\end{align*}
For $K_2$, we set $C_\alpha^l=:T_\alpha[\nabla\nabla^l\u, \d]$, $D_\alpha^l:=\nabla^l T_\alpha[\nabla\u, \d]-C_\alpha^l$, then we have 
\begin{align*}
  K_2&=-\int_{B_2} [C_\alpha^l\cdot \Delta\nabla^l\d\eta^2+D_\alpha^l\cdot \Delta\nabla^l\d\eta^2+C_\alpha^l\cdot (\nabla(\eta^2)\cdot \nabla\nabla^l\d)+D_\alpha^l\cdot (\nabla(\eta^2)\cdot \nabla\nabla^l\d)]dx\\
  &=:K_{21}+K_{22}+K_{23}+K_{24}.
\end{align*}
Now we estimate 
\begin{align*}
  |K_{22}|&\le \frac{1}{32}\int_{B_2}|\nabla^{l+2}\d|^2\eta^2dx+C\int_{B_2}|\nabla\d|^2|\nabla^l\u|^2\eta^2dx+C\int_{B_2}\sum_{j=1}^{l-1}|\nabla^{j}\u|^2|\nabla^{l+1-j}\d|^2\eta^2dx, \\
  |K_{23}|&\le \frac{1}{32}\int_{B_2}|\nabla^{l+1}\u|^2\eta^2dx
  + C\int_{\supp \eta}|\nabla^{l+1}\d|^2dx, \\
  |K_{24}|&\lesssim \int_{\supp \eta} |\nabla^{l+1}\d|^2dx
  +\int_{B_2}|\nabla\d|^2|\nabla^l\u|^2\eta^2dx
  +\int_{B_2}\sum_{j=1}^{l-1}|\nabla^{j}\u|^2|\nabla^{l+1-j}\d|^2\eta^2dx.
\end{align*}
Combine all estimate above, and with the cancellation $I_{51}=K_{21}$, we arrive at
\begin{equation}
\begin{split}
  &\frac{d}{dt}\int_{B_2}\left( |\nabla^l\u|^2+|\nabla^{l+1}\d|^2 \right)\eta^2dx
  +\int_{B_2}\left( |\nabla^{l+1}\u|^2+|\nabla^{l+2}\d|^2 \right)\eta^2dx\\
  &\le C\int_{B_2}(|\u|^2|\nabla^l\u|^2\eta^2 +\sum_{j=1}^{l-1}|\nabla^j\u|^2|\nabla^{l-j}\u|^2\eta^2)dx
  +C\int_{\supp \eta}(|\nabla^l\u|^2+|\nabla^{l+1}\d|^2)dx\\
  &+C\int_{\supp \eta}(|\nabla^l\u|^2+|\nabla^{l+1}\d|^2+|\nabla^{l-1}P^{(1)}|^2+ |\u|^{3}+|P^{(2)}|^{\frac{3}{2}})dx\\
 &+C\int_{B_2}(|\nabla\d|^2|\nabla^{l+1}\d|^2\eta^2+\sum_{j=1}^{l-1}|\nabla^{j+1}\d|^2|\nabla^{l+1-j}\d|^2\eta^2)dx\\
 &+C\int_{B_2} (|\nabla^l F(\d)|^2\eta^2+|\nabla^l S_{\alpha}[\f(\d),\d]|^2\eta^2+|\nabla^{l+1}\f(\d)|^2\eta^2)dx\\
  &+C\int_{B_2}(|\nabla\d|^2|\nabla^l\u|^2\eta^2+|\u|^2|\nabla^{l+1}\d|^2\eta^2 +\sum_{j=1}^{l-1}|\nabla^j\d|^2|\nabla^{l+1-j}\d|^2\eta^2)dx\\
  &+C\int_{B_2}\sum_{j=1}^{l-1}|\nabla^j\u|^2|\nabla^{l+1-j}\d|^2\eta^2dx.
  \label{eqn:HighEn1}
  \end{split}
\end{equation}
By Sobolev-interpolation inequality, we have
\begin{align*}
  &\int_{B_2}|\u|^2|\nabla^l\u|^2\eta^2dx\\
  &\le \left\|\nabla^l \u\eta\right\|_{L^6(B_2)}\left\|\nabla^l \u\eta\right\|_{L^2(B_2)}\left\|\u\right\|_{L^{6}(\supp \eta)}^2\\
  &\le C\left\|\nabla(\nabla^l\u\eta)\right\|_{L^2(B_2)}\left\|\nabla^l\u\eta\right\|_{L^2(B_2)}\left\|\u\right\|_{L^{6}(\supp \eta)}^2\\
  &\le \frac{1}{32}\int_{B_2}|\nabla^{l+1}\u|^2\eta^2dx
  + C\int_{\supp \eta}|\nabla^l\u|^2dx
  +C\left\|\u\right\|_{L^{6}(\supp \eta)}^4\int_{B_2}|\nabla^l\u|^2\eta^2dx, 
  \end{align*}
  \begin{align*}
 & \int_{B_2}|\u|^2|\nabla^{l+1}\d|^2\eta^2dx\\
  &\le \frac{1}{32}\int_{B_2}|\nabla^{l+2}\d|^2\eta^2dx
  +C\int_{\supp \eta}|\nabla^{l+1}\d|^2dx
  +C\left\|\u\right\|_{L^{6}(\supp \eta)}^4\int_{B_2}|\nabla^{l+1}\d|^2\eta^2dx,
  \end{align*}
  \begin{align*}
 & \int_{B_2}|\nabla\d|^2|\nabla^{l}\u|^2\eta^2dx\\
 &\le \frac{1}{32}\int_{B_2}|\nabla^{l+1}\u|\eta^2dx
 +C\int_{\supp \eta}|\nabla^{l}\u|^2dx
 +C\left\|\nabla\d\right\|_{L^{6}(\supp \eta)}^4\int_{B_2} |\nabla^l\u|^2\eta^2 dx, 
 \end{align*}
 \begin{align*}
  &\int_{B_2}|\nabla\d|^2|\nabla^{l+1}\d|^2\eta^2dx\\
  &\le \frac{1}{32}\int_{B_2}|\nabla^{l+2}\d|^2\eta^2dx
  +C\int_{\supp\eta}|\nabla^{l+1}\d|^2dx
  +C\left\|\nabla\d\right\|_{L^{6}(\supp\eta)}^4\int_{B_2}|\nabla^{l+1}\d|^2\eta^2dx.
\end{align*}
For lower order terms, we have that for $1\le j\le l-1$,
\begin{align*}
  \int_{B_2}|\nabla^{l-1}\u|^2|\nabla^{j}\u|^2\eta^2dx
  &\le \left\|\nabla^{l-1}\u\eta\right\|_{L^6(B_2)}^2\left\|\nabla^{j}\u\right\|_{L^3(\supp \eta)}^2\\
  &\le C\left\|\nabla(\nabla^{l-1}\u\eta)\right\|_{L^2(B_2)}^2\left\|\nabla^{j}\u\right\|_{L^3(\supp \eta)}^2\\
  &\le C\left\|\nabla^{j}\u\right\|_{L^3(\supp \eta)}^2\int_{B_2}|\nabla^{l}\u|^2\eta^2dx\\
  &+ C \left\|\nabla^{l-1}\u\right\|_{L^3(\supp \eta)}^2
  \left\|\nabla^{j}\u\right\|_{L^3(\supp \eta)}^2, \\
  \int_{B_2}|\nabla^{l}\d|^2|\nabla^j\u|^2\eta^2dx
  &\le C\left\|\nabla^j\u\right\|_{L^3(\supp \eta)}^2\int_{B_2}|\nabla^{l+1}\d|^2\eta^2dx\\
  &+C\left\|\nabla^l\d\right\|_{L^3(\supp\eta)}^2\left\|\nabla^j\u\right\|_{L^3(\supp \eta)}^2, \\
  \int_{B_2}|\nabla^{l-1}\u||\nabla^{j+1}\d|^2\eta^2dx
  &\le C\left\|\nabla^{j+1}\d\right\|_{L^3(\supp \eta)}^2\int_{B_2}|\nabla^l\u|\eta^2\\
  &+C\left\|\nabla^{l-1}\u\right\|_{L^3(\supp \eta)}^2\left\|\nabla^{j+1}\d\right\|_{L^3(\supp \eta)}^2, \\
  \int_{B_2}|\nabla^l\d|^2|\nabla^{j+1}\d|^2\eta^2dx
  &\le C\left\|\nabla^{j+1}\d\right\|_{L^3(\supp \eta)}^2\int_{B_2}|\nabla^{l+1}\d|^2\eta^2dx\\
  &+C\left\|\nabla^l\d\right\|_{L^3(\supp \eta)}^2\left\|\nabla^{j+1}\d\right\|_{L^3(\supp \eta)}^2, 
  \end{align*}
and for $1\le j, k\le l-2$ that 
\begin{align*}
  \int_{B_2}|\nabla^j\u|^2|\nabla^{k+1}\d|^2\eta^2dx
  &\le C\int_{\supp\eta}|\nabla^j\u|^4dx+C\int_{\supp \eta}|\nabla^{k+1}\d|^4dx.
\end{align*}
Since $|\d|\le M$ in $\PP_2$,  by the calculus inequality for $H^s$ (c.f. \cite[Appendix]{KainermanSergiu1982}), 
we have for $-4\le t\le 0$, 
\begin{align*}
  \left\|\nabla^l F(\d)\right\|_{L^2(\supp \eta)}&\lesssim \left\|\nabla^l\d\right\|_{L^2(\supp \eta)}, \\
  \left\|\nabla^lS_\alpha[\f(\d), \d]\right\|_{L^2(\supp \eta)}&\lesssim \left\|\nabla^l \d\right\|_{L^2(\supp \eta)}, \\
  \left\|\nabla^{l+1}\f(\d)\right\|_{L^2(\supp \eta)}&\lesssim \left\|\nabla^{l+1}\d\right\|_{L^2(\supp\eta)}.
\end{align*}

Put all these estimates together, we arrive at
\begin{equation}
  \begin{split}
    &\frac{d}{dt}\int_{B_2}\left( |\nabla^l\u|^2+|\nabla^{l+1}\d|^2 \right)\eta^2dx
    +\int_{B_2}\left( |\nabla^{l+1}\u|^2+|\nabla^{l+2}\d|^2 \right)\eta^2dx\\
    &\le C\int_{\supp \eta} [|\nabla^l\u|^2+|\nabla^{l+1}\d|^2+|\nabla^l\d|^2+\sum_{j=1}^{l-2}(|\nabla^j\u|^4+|\nabla^{j+1}\d|^4]dx)\\
    &+C\int_{\supp\eta}(|\u|^{3}+|\nabla^{l-1}P^{(1)}|^2+|P^{(2)}|^{\frac{3}{2}})dx\\
    &+C \Big(\left\|\nabla^{l-1}\u\right\|_{L^3(\supp \eta)}^4+\left\|\nabla^l \d\right\|_{L^3(\supp \eta)}^4+\sum_{j=1}^{l-1}\big(\left\|\nabla^j\u\right\|_{L^3(\supp\eta)}^4+\left\|\nabla^{j+1}\d\right\|_{L^3(\supp \eta)}^4\big)\Big)\\
    &+C\Big(\left\|(\u, \nabla\d)\right\|_{L^{6}(B_2)}^{4}+\sum_{j=1}^{l-1}\|(\nabla ^j\u, \nabla^{j+1}\d)\|_{L^3(B_2)}^2\Big)\int_{B_2}(|\nabla^{l}\u|^2+|\nabla^{l+1}\d|^2)\eta^2dx.
  \end{split}
  \label{eqn:highNonEnergy}
\end{equation}
Now let $\eta\in C_0^\infty(B_{1+2^{-(l+1)}+5^{-(l+1)}})$ be a cut-off function of $B_{1+2^{-(l+1)}+10^{-(l+1)}}$. We can apply the Gronwall's inequality to \eqref{eqn:highNonEnergy}, together with \eqref{4.79}-\eqref{eqn:Lpbound} to get 
  \begin{equation}
\begin{split}
&\sup_{-\big( 1+2^{-(l+1)+10^{-(l+1)}} \big)^2\le t\le 0}\int_{B_{1+2^{-(l+1)}+10^{-(l+1)}}}(|\nabla^l\u|^2+|\nabla^{l+1}\d|^2)dx\\
&\quad+\int_{\PP_{1+2^{-(l+1)}+10^{-(l+1)}}}\left( |\nabla^{l+1}\u|^2+|\nabla^{l+2}\d|^2 \right)dxdt\\
&\le C(l)\varepsilon_1.
\end{split}
\label{4.94}
\end{equation}

Recall that $\nabla^l P$ satisfies
\begin{equation}
\begin{split}
    -\Delta \nabla^{l} P&=\dv^2\Big[\nabla^{l}\Big(\u\otimes\u+\nabla\d\odot\nabla\d-\frac{1}{2}|\nabla \d|^2I_3\\
&-(F(\d)I_3-\fint_{\PP_2}F(\d)I_3)
+S_\alpha[\Delta\d-\f(\d), \d]+\fint_{\PP_2}S_\alpha[\f(
\d), \d]\Big)\Big].
\end{split}
\end{equation}
Then by the Calder\'{o}n-Zygmund theory and \eqref{4.79}-\eqref{eqn:Lpbound}, \eqref{4.94} we can show
\begin{equation}
\begin{split}
  &\int_{\PP_{1+2^{-(l+1)}}}|\nabla^l P|^{\frac{3}{2}}dxdt\le C(l)\varepsilon_1.
  \end{split}
  \label{}
\end{equation}
This yields that the conclusion holds for $k=l$. Thus the proof is complete. 
\end{proof}

\section{Partial regularity}
As a consequence of Lemma \ref{lemma:HigherOrderRegularity}, we get the following regularity criteria for \eqref{eqn:ELmodel1}: 
\begin{corollary}
	For a suitable weak solution $(\u, \d, P)$ to \eqref{eqn:ELmodel1}, if $z\in {\R^3}\times (0, \infty)$ satisfies
	\begin{equation}
	\left\{
	\begin{array}{l}
	\sup_{0<r<\delta}|\d_{z, r}|<\infty, \\
	\liminf_{r\to 0+}\Phi(z, r)=0, 
	\end{array}
	\right.
	\label{eqn:regpoint}
	\end{equation}
	Then there exists $\delta_1>0$ such that $(\u, \d)\in C^\infty(\PP_{\delta_1}(z))$.
	\label{cor:regcrateria}
\end{corollary}
The following Lemma is well-known, see \cite{GiaquintaGiusti1973}. 
\begin{lemma}
  Let $\d$ be a function in $L^{6}({\R^3}\times(0,\infty))$, and let $z=(x, t)\in {\R^3}\times(0, \infty)$ such that 
  \begin{equation}
    \fint_{\PP_r(z)}|\d-\d_{z, r}|^{6}dxdt\le Cr^{\delta}
    \label{}
  \end{equation}
  for some $\delta>0$ and some $C$ depending on $\d$ and $z$. Then $\lim_{r\to 0}\d_{z, r}$ exists, and is finite.
  \label{lemma:dalp}
\end{lemma}
Next we will  control the oscillation of $\d$. For $0<T\le\infty$, denote  $Q_T={\R^3}\times(0,T)$.
Recall the fractional parabolic Sobolev space $W_p^{1, \frac{1}{2}}(Q_T)$, $1\le p<\infty$,
contains all $f$'s satisfying
\begin{equation*}
  \left\|f\right\|_{W_p^{1,  \frac{1}{2}}(Q_T)}=\left\|f\right\|_{L^p(Q_T)}+\|f\|_{\dot{W}^{1, \frac{1}{2}}_p(Q_T)}<\infty,
  \end{equation*}
  where 
  \begin{equation*}
  \|f\|_{\dot{W}^{1, \frac{1}{2}}_p(Q_T)}:=\big(\int_{Q_T}|\nabla f|^pdtdx+ \int_{\R^3}\ \int_0^T\int_0^T\frac{|f(x,t)-f(x,s)|^p}{|t-s|^{1+\frac{p}{2}}}dtds dx \big)^{\frac{1}{p}}.
\end{equation*}

From the global energy estimate \eqref{eqn:globaleng} and the Sobolev embedding theorem, we have
\begin{equation}
    (\u,\nabla \d) \in (L_t^\infty L_x^2 \cap L_t^2 H_x^1\cap L_t^{\frac{10}{3}}L_x^{\frac{10}{3}})(Q_T), \ 
    \d \in L_t^{10} L_x^{10}(Q_T).
  \label{}
\end{equation}
It follows that
\begin{equation*}
  \pa_t\d=\Delta\d-\f(\d)-\u\cdot \nabla\d+T_\alpha[\nabla\u,\d]\in L^{\frac{5}{3}}(Q_T).
\end{equation*}
From the fractional Galiardo-Nirenberg inequality \cite{Brezis2001fraction, Brezis2018Galiardo}, we get $\d\in W^{1,\frac{1}{2}}_{\frac{20}{7}}(Q_T) $, and
\begin{equation*}
  \left\|\d\right\|^2_{W^{1,\frac{1}{2}}_\frac{20}{7}(Q_T)}\le C \left\|\d\right\|_{L^{10}(Q_T)}\left\|(\pa_t \d,\nabla\d)\right\|_{L^{\frac{5}{3}}(Q_T)}+C\left\|\d\right\|_{L_t^{\frac{20}{7}}W_x^{1, \frac{20}{7}}(Q_T)}^2
  <\infty. 
\end{equation*}
Then the parabolic Sobolev-Poincar\'{e} inequality yields 
\begin{equation*}
  \begin{split}
    &\big( \fint_{\PP_r(z)}|\d-\d_{z,r}|^p dxdt \big)^{\frac{1}{p}}\\
    &\le C\Big[  r^{\frac{20}{7}-5}\int_{\PP_r(z)}|\nabla\d|^{\frac{20}{7}}
+ r^{\frac{20}{7}-5}\int_{B_r(x)}\int_{t-r^2}^t\int_{t-r^2}^t \frac{|\d(x,s_1)-\d(x,s_2)|^{\frac{20}{7}}}{|s_1-s_2|^{1+\frac{10}{7}}}ds_1ds_2dx\Big]^{\frac{7}{20}}.
     \end{split}
\end{equation*}
where $p=\frac{5\cdot\frac{20}{7}}{5-\frac{20}{7}}=\frac{20}{3}>6.$
Hence by H\"{o}lder inequality we have that 
\begin{equation}
\begin{split}
 & \big( \fint_{\PP_r(z)}|\d-\d_{z,r}|^6 dxdt \big)^{\frac{1}{6}}\le \big( \fint_{\PP_r(z)}|\d-\d_{z,r}|^{\frac{20}{3}} dxdt \big)^{\frac{3}{20}}\\
 &\le C\Big[  r^{\frac{20}{7}-5}\int_{\PP_r(z)}|\nabla\d|^{\frac{20}{7}}
  + r^{\frac{20}{7}-5}\int_{B_r(x)}\int_{t-r^2}^t\int_{t-r^2}^t \frac{|\d(x,s_1)-\d(x,s_2)|^{\frac{20}{7}}}{|s_1-s_2|^{1+\frac{10}{7}}}ds_1ds_2dx\Big]^{\frac{7}{20}}.
  \end{split}
  \label{eqn:SobolevPoincare}
\end{equation}

\begin{proof} [Proof of Theorem \ref{thm:main}] Define 
  \begin{equation*}
    \Sigma=\left\{ z\in{\R^3}\times(0, \infty): \liminf_{r\to0}\Phi(z, r)>\varepsilon_2^{6}
    \ {\rm{or}}\  \liminf_{r\to 0} |\d_{z,r}|=\infty\right\}.
  \end{equation*}
  It follows from Corollary \ref{cor:regcrateria} that $\Sigma$ is closed and $(\u, \d)\in C^\infty( {\R^3}\times(0, \infty)\setminus \Sigma)$. From \eqref{eqn:SobolevPoincare} and Lemma \ref{lemma:dalp}, we know that
  $\Sigma\subset\cap_{\sigma>0}\mathcal{S}_\sigma$, where $\mathcal{S}_\sigma$ is defined by 
  \begin{equation*}
    \begin{split}
      &\mathcal{S}_\sigma=\Big\{z\in Q_T:\liminf_{r\to 0}\big[r^{-\frac{5}{3}}\int_{\PP_r(z)}\big( |\u|^{\frac{10}{3}}+|\nabla\d|^{\frac{10}{3}} \big)dxdt+\big( r^{-\frac{5}{3}}\int_{\PP_r(z)}|P|^{\frac{5}{3}}dxdt \big)^{2}\big]>0,  \text{ or }\\
   &  \liminf_{r\to 0}r^{-\frac{15}{7}-\sigma}\big(\int_{\PP_r(z)}|\nabla\d|^{\frac{20}{7}}dxdt+\int_{B_r(x)}\int_{t-r^2}^t\int_{t-r^2}^t \frac{|\d(x,s_1)-\d(x,s_2)|^{\frac{20}{7}}}{|s_1-s_2|^{1+\frac{10}{7}}}ds_1ds_2 dx\big)>0 \Big\}.
    \end{split}
  \end{equation*}
  For the last integral, we have that 
  \begin{equation*}
    f(x, s_1, s_2)=\frac{|\d(x, s_1)-\d(x, s_2)|^{\frac{20}{7}}}{|s_1-s_2|^{1+\frac{10}{7}}}\in L^1(\R^3\times(0,T)\times(0,T)).
  \end{equation*}
  Let $\tilde{\delta}$ be the metric on $\R^3\times\R\times\R$:
  \begin{equation*}
    \tilde{\delta}(\xi_1, \xi_2)=\max\left\{ |x_1-x_2|, \sqrt{|t_1-t_2|}, \sqrt{|s_1-s_2|} \right\}, \quad \forall \xi_i=(x_i, t_i, s_i)\in \R^3\times\R\times\R.
  \end{equation*}
A standard covering argument implies that 
\begin{equation*}
  \widetilde{\mathcal{P}}^{\frac{15}{7}+\sigma}\big\{ (x,s,t)\in\R^3\times(0,T)\times(0, T):\liminf_{r\to 0+}r^{-\frac{15}{7}-\sigma} \int_{B_r(x)}\int_{s-r^2}^s\int_{t-r^2}^t f(\xi) d\xi>0 \big\}=0,
\end{equation*}
where $\widetilde{\mathcal{P}}^k$ denotes the $k$-dimensional Hausdorff measure on $\R^3\times\R_+\times\R_+$ with respect to the metric $\tilde{\delta}$. 

Since the map $T(x,t)=(x,t,t):\R^3\times\R\to\R^3\times\R\times\R$
is an isometric embedding of $(\R^3\times\R, \delta)$ into $(\R^3\times\R\times \R, \tilde{\delta})$,
we have that
\begin{equation}
  \begin{split}
  &\mathcal{P}^{\frac{15}{7}+\sigma}\left( \left\{ (x,t)\in Q_T:\liminf_{r\to 0+}r^{-\frac{15}{7}-\sigma} \int_{B_r(x)}\int_{t-r^2}^t\int_{t-r^2}^{t}f(\xi) d\xi>0 \right\} \right)\\
  &=\widetilde{\mathcal{P}}^{\frac{15}{7}+\sigma}\left( T\left[ \left\{ (x,t)\in Q_T:\liminf_{r\to 0+}r^{-\frac{15}{7}-\sigma}\int_{B_r(x)}\int_{t-r^2}^{t}\int_{t-r^2}^t f(\xi) d\xi>0 \right\} \right]\right)\\
  &=\widetilde{\mathcal{P}}^{\frac{15}{7}+\sigma}\left(\left\{ (x,t,t)\in Q_T\times (0,T):\liminf_{r\to 0+}r^{-\frac{15}{7}-\sigma}\int_{B_r(x)}\int_{t-r^2}^{t}\int_{t-r^2}^t f(\xi)d\xi>0 \right\}\right)\\
  &\le  \widetilde{\mathcal{P}}^{\frac{15}{7}+\sigma}\left(\left\{ (x,s,t)\in Q_T\times(0, T):\liminf_{r\to 0+}r^{-\frac{15}{7}-\sigma} \int_{B_r(x)}\int_{s-r^2}^s\int_{t-r^2}^t f(\xi)d\xi>0 \right\}\right)\\
  &=0.
\end{split}
  \label{eqn:measure1}
\end{equation}
Again, by a simple covering argument we can show
\begin{equation}
  \mathcal{P}^{\frac{15}{7}+\sigma}\Big( \big\{ z\in Q_T: r^{-\frac{15}{7}-\sigma}\int_{\PP_r(z)}|\nabla\d|^{\frac{20}{7}}dxdt>0 \big\} \Big)=0, 
  \label{eqn:measure2}
\end{equation}
and
\begin{equation}
  \mathcal{P}^{\frac{5}{3}}\Big( \big\{ z\in Q_T:\lim_{r\to 0}r^{-\frac{5}{3}}\int_{\PP_r(z)}(|\u|^{\frac{10}{3}}+|\nabla\d|^{\frac{10}{3}})dxdt+\big( r^{-\frac{5}{3}}\int_{\PP_r(z)}|P|^{\frac{5}{3}} \big)^2>0 \big\} \Big)=0.
  \label{eqn:measure3}
\end{equation}
It follows from \eqref{eqn:measure1}, \eqref{eqn:measure2} and \eqref{eqn:measure3} that $\P^{\frac{15}{7}+\sigma}(\mathcal{S}_\sigma)=0$ so that $\P^{\frac{15}{7}+\sigma}(\Sigma)=0, \forall \sigma>0.$
\end{proof}



\end{document}